\documentclass[11pt]{amsart}
\usepackage[top=2cm, bottom=2cm, left=2.7cm, right=2.7cm]{geometry}
\usepackage{setspace}
\usepackage{amsmath,amssymb,amsthm}
\usepackage{mathrsfs}  
\usepackage{bbm}        
\usepackage{esint}     
\usepackage{dsfont}
\usepackage[dvipsnames]{xcolor}
\usepackage{cancel}
\usepackage{todonotes}

\usepackage[utf8]{inputenc}
\usepackage[T1]{fontenc}
\usepackage[english]{babel}
\usepackage[noadjust]{cite}

\usepackage{calc}
\usepackage[inline]{enumitem}
\usepackage[normalem]{ulem}
\usepackage{url}

\usepackage[bookmarks, bookmarksnumbered, pdfstartview={XYZ null null 1.00}]{hyperref}

\newcommand{\theoname}{Theorem}
\newcommand{\lemmname}{Lemma}
\newcommand{\coroname}{Corollary}
\newcommand{\propname}{Proposition}
\newcommand{\definame}{Definition}

\newcommand{\remkname}{Remark}
\newcommand{\explname}{Example}

\theoremstyle{plain}
\newtheorem{theorem}{\theoname}[section]
\newtheorem{lemma}[theorem]{\lemmname}

\newtheorem{proposition}[theorem]{\propname}

\theoremstyle{definition}
\newtheorem{definition}[theorem]{\definame}
\newtheorem{remark}[theorem]{\remkname}

\allowdisplaybreaks

\def\Xint#1{\mathchoice{\XXint\displaystyle\textstyle{#1}}%
  {\XXint\textstyle\scriptstyle{#1}}%
  {\XXint\scriptstyle\scriptscriptstyle{#1}}%
  {\XXint\scriptscriptstyle\scriptscriptstyle{#1}}%
  \!\int}
\def\XXint#1#2#3{{\setbox0=\hbox{$#1{#2#3}{\int}$}
  \vcenter{\hbox{$#2#3$}}\kern-.5\wd0}}
\def\fint{\Xint-}

\newcommand{\dHn}{\,\dd\mathcal{H}^{n-1}}
\newcommand{\dsty}{\displaystyle}
\newcommand{\eqdef}{\overset{\text{def.}}{=}}

\def\dd{{\rm d}}
\newcommand{\norm}[1]{\left\lVert#1\right\rVert}

\def\ddiv{{\mathop{\rm div}}}

\newcommand{\mres}{\mathbin{\vrule height 1.6ex depth 0pt width
		0.13ex\vrule height 0.13ex depth 0pt width 1.3ex}}

\newcommand\TV{\mathop{\rm TV}}
\newcommand\BV{\mathop{\rm BV}}
\newcommand\Per{\mathop{\rm Per}}

\newcommand\supp{\mathop{\rm supp}}
\newcommand\diam{\mathop{\rm diam}}

\newcommand\SV{\mathop{\rm SV}}

\title[Effective Quantitative stability of Faber-Krahn]{On the effective sharp
stability of the Faber-Krahn inequality}

\author{Andr\'e Guerra}
\email{adblg2@cam.ac.uk}
\address{Department of Pure Mathematics and Mathematical Statistics (DPMMS) at the University of Cambridge - Centre for Mathematical Sciences, Wilberforce Road, Cambridge, CB3 0WB, United Kingdom}

\author{Jo\~ao Miguel Machado}
\email{joao-miguel.machado@ceremade.dauphine.fr}
\address{Lagrange Mathematical and Computational Center\\
103 rue de Grenelle\\
Paris, 75007, France}

\author{Jo\~ao P. G. Ramos}
\email{joao.ramo@impa.br}
\address{Instituto Nacional de Matem\'atica Pura e Aplicada (IMPA) \\ 
Estrada Dona Castorina 110 \\
22460-320, Jardim Bot\^anico - Rio de Janeiro, RJ, Brasil}

\date{\today}

\begin{document}
\raggedbottom

\begin{abstract}
We prove the sharp quantitative Faber--Krahn inequality with a computable dimensional constant, resolving a problem posed by Brasco and De Philippis~\cite[Open Problem~2]{brasco2017spectral}.

In contrast with the known proof of the sharp estimate, our argument avoids indirect compactness arguments and selection principles. It runs through a new level-set proof of the Saint--Venant inequality, using that the torsional deficit controls both the average isoperimetric deficit and the oscillation of the gradient magnitude along the boundaries of level sets. An effective local strong isoperimetric estimate for confined sets then yields the result.

 \bigskip


\\ 
\end{abstract}

\maketitle

\section{Introduction}
Let $\Omega\subset\mathbb R^n$, $n\ge2$, be an open set of finite positive
measure. Its first Dirichlet eigenvalue is
\[
  \lambda_1(\Omega)
  \eqdef
  \inf_{
    \substack{
        v\in W_0^{1,2}(\Omega)
    } \, 
  }
  \left\{
  \displaystyle
  \int_\Omega |\nabla v|^2\dd x: 
  \quad 
  \norm{v}_{L^2(\Omega)} = 1
  \right\}.
\]
The Faber--Krahn inequality states that
\[
  \lambda_1(\Omega)\ge \lambda_1(B_\Omega),
\]
where $B_\Omega$ is any ball with $|B_\Omega|=|\Omega|$. Equality holds only
when $\Omega$ agrees almost everywhere with a ball.
 This inequality, a cornerstone result in modern geometric analysis and calculus of variations, was originally conjectured by Rayleigh and proved
independently by Faber and Krahn.

In spite of the fact that the Faber--Krahn inequality is, by now, a standard result in calculus of variations, with its proof via the P\'olya-Szeg\H{o} being, by now, as classical as the result itself, the question of its sharp stability has a much more recent history.

In order to state it, we first recall that the natural scaling of the Dirichlet eigenvalue is given by 
\[
    \lambda_1(t \Omega) = t^{-2}\lambda_1(\Omega).
\]
Hence, to measure oscillations of the Faber-Krahn inequality, it is useful to write it in the following equivalent \textit{scale invariant form} 
\[
    |\Omega|^{2/n} \lambda_1(\Omega)
    \ge 
    |B|^{2/n} \lambda_1(B),
\]
which is valid for any euclidean ball $B$. On the other hand, the distance of $\Omega$ to the family of balls can be quantified up to changes of scale by its \emph{Fraenkel asymmetry}
\[
  \mathcal A(\Omega)
  \eqdef
  \inf_{z\in\mathbb R^n}
  \frac{|\Omega\mathbin\Delta B_\Omega(z)|}{|\Omega|}.
\]

The question of making the Faber--Krahn inequality quantitative is then equivalent to asking how small $\mathcal{A}(\Omega)$ has to be, especially when compared to the natural deficit
\begin{equation*}
  \delta_{\rm FK} (\Omega) \eqdef 
  |\Omega|^{2/n} \lambda_1(\Omega) - |B|^{2/n} \lambda_1(B).
\end{equation*}
This question was first addressed for arbitrary open sets in every dimension by Fusco--Maggi--Pratelli~\cite{FuscoMaggiPratelli2009Eigenvalue}, who proved that
\[
\mathcal{A}(\Omega)^4 \leq C(n) \delta_{\rm FK}(\Omega),
\]
where $C(n) > 0$ is a computable dimensional constant. Brasco--De Philippis~\cite[Theorem~2.10]{brasco2017spectral} later improved the exponent from $4$ to $3$. Finally, Brasco--De Philippis--Velichkov~\cite{brasco2015faber} proved that
\begin{equation}\label{eq.intro_quant_FK}
\mathcal{A}(\Omega)^2 \le C \, \delta_{\rm FK}(\Omega),
\end{equation}
for some finite constant $C>0$. The exponent $2$ is optimal, as is seen from smooth perturbations of the ball. In spite of the striking nature of \eqref{eq.intro_quant_FK}, the constant obtained in \cite{brasco2015faber} is \emph{non-explicit}, as their proof uses a selection and compactness argument, and derives the sharp version thereof by contradiction. We also refer the reader to the works of Allen--Kriventsov--Neumayer~\cite{Allen2021,AllenKriventsovNeumayer2023,AllenKriventsovNeumayer2025} and Fleschler--Tolsa--Villa~\cite{FleschlerTolsaVilla2024}. For the first eigenvalue of the $p$-Laplacian, Bhattacharya~\cite{Bhattacharya2001} obtained earlier asymmetry estimates, and Fusco--Zhang~\cite{FuscoZhang2017} proved the sharp quantitative inequality for every $p>1$, extending~\cite{brasco2015faber}. Ideas from the latter work have also informed quantitative Faber--Krahn-type results for the Robin, fractional, and Gaussian Laplacians~\cite{BucurFeroneNitschTrombetti2018,BrascoCintiVita2020,CarbottiCitoLaMannaPallara2024}, the fundamental tones of free plates~\cite{BuosoChasmanProvenzano2018}, and the short-time Fourier transform~\cite{GomezGuerraRamosTilli2024}. Closely related developments include the sharp quantitative $p$-isocapacitary inequality~\cite{Mukoseeva2023}, quantitative estimates for $p$-torsion on convex sets~\cite{AmatoMasielloPaoliSannipoli2023}, and stability estimates for the Dirichlet spectrum near the ball~\cite{BucurLamboleyNahonPrunier2026}.

The purpose of this article is to prove~\eqref{eq.intro_quant_FK} with a
computable dimensional constant. We use the term \emph{effective sharp
quantitative Faber--Krahn inequality} for this conclusion. Here effective
means that the constant can be traced and computed explicitly through the proof; we, however, do not try to optimize
its numerical value.

More precisely, a dimensional constant is called computable if an algorithm produces certified rational upper and lower bounds
converging to it. In particular, every smallness threshold below is obtained from previously fixed computable constants by arithmetic operations or finite searches; no compactness argument is used to select it.

\begin{theorem}\label{thm.sharp_FK}
For every $n\ge2$ there is a computable constant $C_{\rm FK}(n)>0$ such
that every open set $\Omega\subset\mathbb R^n$ satisfies
\[
  \mathcal A(\Omega)^2
  \le C_{\rm FK}(n)\delta_{\rm FK} (\Omega).
\]
\end{theorem}

In particular, assuming for simplicity that $|\Omega| = \omega_n$, for every $\varepsilon>0$, the condition
$\mathcal A(\Omega)\ge\varepsilon$ gives the explicit bound
\[
  \lambda_1(\Omega)-\lambda_1(B_1)
  \ge \frac{\varepsilon^2}{C_{\rm FK}(n)}.
\]
Thus the result provides explicit rigidity thresholds in every argument in
which a lower bound for the asymmetry is known.

The proof's first step passes through the Saint--Venant inequality. Let $u=u_\Omega$ solve
\[
  -\Delta u=1\quad\text{in }\Omega,
  \qquad
  u=0\quad\text{on }\partial\Omega,
\]
and define the torsional rigidity by
\[
  T(\Omega)\eqdef\int_\Omega u\dd x,
\]
where this time, the scaling law is given by $T(t\Omega) = t^{n+2}T(\Omega)$, see~\ref{lemma.torsion_scaling}. Therefore, among sets with fixed volume, the Saint--Venant inequality states that the torsion is maximized by the corresponding balls, which can be expressed in the scale invariant form as 
\[
    \delta_{\SV}(\Omega)
    \eqdef 
    |B|^{-\frac{n+2}{n}}T(B)-|\Omega|^{-\frac{n+2}{n}}T(\Omega)
    \ge 0, 
\]
where $B$ is any euclidean ball thanks to the scale invariance.

The link between the Faber--Krahn and Saint--Venant is then established via the Kohler--Jobin inequality~\cite{kohler-jobin1978methode,brasco2014torsional}, which states that for any open set $\Omega \subset \mathbb{R}^n$ it holds that
\[
  \lambda_1(\Omega)T(\Omega)^{\frac{2}{n+2}}
  \ge
  \lambda_1(B_1)T(B_1)^{\frac{2}{n+2}},
\]
notice that the scaling factors of the two quantities cancel each other on both sides. Therefore, our main target will actually be proving the following result:
\begin{theorem}\label{thm.sharp_Saint_Venant_unbounded}
For every $n\ge2$ there is a computable constant $C_{\rm SV}(n)>0$ such
that every open set $\Omega\subset\mathbb R^n$ with $|\Omega|=\omega_n$
satisfies
\begin{equation}\label{eq.intro_quant_SV}
  \mathcal A(\Omega)^2
  \le C_{\rm SV}(n)\delta_{\rm SV}(\Omega).
\end{equation}
\end{theorem}

Indeed, once Theorem~\ref{thm.sharp_Saint_Venant_unbounded} is established the sharp stability of Faber--Krahn with a computable constant follows. 

\begin{proof}[Proof of Theorem~\ref{thm.sharp_FK}]
The Kohler--Jobin inequality gives
\[
 T(\Omega)\le T(B_1)
 \left(\frac{\lambda_1(B_1)}{\lambda_1(\Omega)}\right)^{(n+2)/2},
\]
thanks to the scale invariance we can take the unit ball for both quantities $\lambda_1(B_1)$ and $T(B_1)$. Since $1-(1+x)^{-(n+2)/2}\le (n+2)x/2$ for $x\ge0$,
\[
 \delta_{\rm SV}(\Omega)
 \le\frac{n+2}{2}\frac{T(B_1)}{\lambda_1(B_1)}
 \delta_{\rm FK}(\Omega).
\]
Theorem~\ref{thm.sharp_Saint_Venant_unbounded} proves the result with
\[
 C_{\rm FK}(n)
 =\frac{n+2}{2}\frac{T(B_1)}{\lambda_1(B_1)}C_{\rm SV}(n).
\]
This constant is computable: $T(B_1)=\omega_n/[n(n+2)]$ and
$\lambda_1(B_1)=j_{n/2-1,1}^2$, is given by the first positive zero o the Bessel function of first kind, which admits computable numerical approximations.
\end{proof}

Our proof of Theorem~\ref{thm.sharp_Saint_Venant_unbounded} starts from an exact level-set identity.
For
\[
  E_t\eqdef\{u>t\},\qquad m(t)\eqdef|E_t|,
\]
set
\[
  \alpha(t)\eqdef\int_{\partial^*E_t}\frac{1}{|\nabla u|}\dHn,
  \qquad
  \beta(t)\eqdef\int_{\partial^*E_t}|\nabla u|\dHn,
  \qquad
  c_n\eqdef n^2\omega_n^{2/n},
\]
and
\[
  D_H(t)\eqdef\alpha(t)\beta(t)-\Per(E_t)^2,
  \qquad
  D_I(t)\eqdef\Per(E_t)^2-c_n\,m(t)^{2-2/n}.
\]
We prove
\begin{equation}\label{eq.intro_torsion_identity}
  \delta_{\SV}(\Omega)
  =\int_0^{\|u\|_\infty}
  \frac{D_H(t)+D_I(t)}{c_n\,m(t)^{1-2/n}}\dd t.
\end{equation}
The nonnegativity of $D_H$ follows from the Cauchy--Schwarz inequality, while
that of $D_I$ follows from the isoperimetric inequality. Identity
\eqref{eq.intro_torsion_identity} therefore gives a direct level-set proof of
the Saint--Venant inequality and separates its two sources of deficit.

The main reason the proofs from~\cite{FuscoMaggiPratelli2009Eigenvalue,brasco2017spectral} yield exponents $3$ and $4$, and hence fail to produce a sharp result, is that at some point they use a Chebyshev inequality in order to select a certain level set with controlled asymmetry. This, however, almost always induces a loss, as on average all level sets must satisfy much better asymmetry.

Therefore, to obtain the sharp quantitative estimate, we parametrize the level sets by
$|E_\rho|=\omega_n\rho^n$ and consider
\[
  q(\rho)\eqdef
  \inf_{z\in\mathbb R^n}|E_\rho\mathbin\Delta B_\rho(z)|.
\]
The idea is to show that the \emph{variation} of the function $q$ is small in terms of $\delta_{\SV}(\Omega)$. This induces a much stronger level-selection mechanism. Once that is done, we basically reuse the device used in the older proofs to conclude.

In order to bound the variation of $q$, let $z_\rho$ be a minimizing center for its definition. If $V_\rho$ denotes, roughly, a function proportional to the inverse of the magnitude of $\nabla u$ on level sets of $u$, and $H_{z,\rho}$ detects how much the normal at that same level set deviates from the normal at the circle, then the upper right derivative $D^+q$ satisfies
\[
  D^+q(\rho)
  \le \frac{n}{\rho} q(\rho)
  + (\text{distance between } V_\rho \text{ and } H_{z,\rho} \text{ on the level set)}.
\]
We then split the distance on the right-hand side above into two contributions. The first is an \emph{oscillation term}, which measures how much $V_{\rho}$ varies. As it turns out, the oscillation of $V_\rho$ is controlled by $D_H$ almost directly from the definition.

We are then left with estimating how much the function $H_{z,\rho}$ varies. As perhaps a surprise, this can be precisely estimated: the strong quantitative
isoperimetric inequality of Fusco--Julin~\cite{fusco2014strong} controls the
remaining normal term by $D_I$. While the global proofs in \cite{fusco2014strong,hensel2026quantitative}
use compactness, we only need the perturbative estimate for confined sets proved directly by Hensel and Laux~\cite[Theorem~2]{hensel2026quantitative}. Therefore, in Proposition~\ref{prop.FJ-strong} below we verify that its smallness threshold can be chosen computably.

Finally, we note that many of the arguments mentioned above need the boundedness assumption $\Omega\subset B_R$. We first carry out our program in that case and, by using truncation methods originally available in \cite{brasco2015faber}, together with the global, nonsharp estimate
\[
  \mathcal A(\Omega)^3\le C(n)\delta_{\SV}(\Omega),
\]
we remove these conditions at the end. The
Kohler--Jobin inequality then concludes the proof of the effective sharp
quantitative Faber--Krahn inequality.

 Section 2 contains the required facts on torsion functions and finite-perimeter
sets, the exact identity relating $\delta_{\SV}(\Omega)$ to $D_H$ and $D_I$,
and the effective nonsharp estimate with exponent $3$. Section 3 proves the
effective sharp Saint--Venant inequality for bounded domains. Section 4
removes the boundedness hypothesis. Finally, Section 5 contains a generalization and a discussion of possible alternative routes to effectivity.

\section*{Acknowledgments}

J. P. G. R.  was supported by the FCT through project SHADE
(2023.17881.ICDT), DOI \texttt{10.54499/}\allowbreak
\texttt{2023.17881.ICDT}. He was also supported  by FAPERJ through  JCNE grant
no.~SEI-260003/\allowbreak020475/2025, and by Instituto Serrapilheira through
grant Serra-R-2510-59765.

\section*{LLM Usage}

Large Language Models played a non-trivial role in the discovery process of the proof below. The original idea for this manuscript started as a discussion between the first and third named authors, where they observed that, in none of the proofs, the fact that one has quantitative control on the variation of the gradient on the boundary of level sets is properly used. Part of the inspiration came from their joint work on the stability of the Faber--Krahn inequality for the short-time Fourier transform~\cite{GomezGuerraRamosTilli2024}, whose proof uses simultaneously all the analytic and geometric information available in that problem.

By interacting with GPT 5.4, it was suggested to the authors by the LLM that the main issue with most of the `simple' proofs of non-sharp stability is choosing a level set sufficiently close to the boundary in which the deficit is comparable to the original one. This sparked the idea of looking at the variation of the asymmetries. 

Another extensive round of interactions -- with several failed routes being explored and discarded with the aid of GPT 5.4 and 5.5 -- generated the fact that the derivative of the asymmetry of the level sets \emph{exists}, as that is a Lipschitz function, and it can be controlled by quantities directly related to the problem. It was then suggested by GPT 5.5, and later confirmed by us, that the variation in question might be related to the recent Fusco--Julin inequality (Lemma \ref{lemma.estimate_moving_ball_derivative}). The final quantitative part, which is the effective version of Fusco--Julin for small deficit (\cite{hensel2026quantitative}; see Proposition \ref{prop.FJ-strong}), was then found by us. 

The final writeup of this manuscript was human-made, with AI tools only being employed in typographical fixes and editing, as well as in sourcing many of the references. The authors have written and checked all proofs, and they take full responsibility for the contents of this manuscript.

\section{Preliminaries and the torsion level-set proof of the Saint-Venant inequality}

We start with the qualitative Saint--Venant inequality. Let $\Omega$ be an
open subset of $\mathbb R^n$. Its torsion function $u=u_\Omega$ is the unique
weak solution of
 \begin{equation}\label{eq.torsion_function}
  \left\{
  \begin{aligned}
    -\Delta u &= 1 && \text{in }\Omega\\
    u &= 0 && \text{on }\partial\Omega .
  \end{aligned}
  \right.
\end{equation}
We recall that the torsion functional can be defined as
\begin{equation}
    T(\Omega) = \int_{\Omega} u_\Omega \dd x.
\end{equation}
It also admits a dual formulation; defining
\begin{equation}\label{eq.torsion_dual_formulation}
     \mathcal{E}(\Omega)\eqdef
     \min_{v\in W_0^{1,2}(\Omega)}
     \left\{\frac{1}{2}\int_\Omega|\nabla v|^2\dd x-\int_\Omega v\dd x\right\}
     \text{ we have } 
     \mathcal{E}(\Omega) = - \frac{1}{2} T(\Omega). 
\end{equation}
Throughout this article, we will privilege the formulation via the torsion function~\eqref{eq.torsion_function} to study the rigidity properties of its level sets. But regardless of the formulation, the Saint--Venant inequality states that  torsion is maximized by balls:
\[
   T(\Omega)\le T(B_\star),
\]
where $B_\star$ is a ball with $|B_\star|=|\Omega|$. We first compute the
torsion of a ball and record its scaling.

\begin{lemma}\label{lemma.torsion_scaling}
  The following identities hold
  \begin{enumerate}
    \item $T(B_1) = \displaystyle \frac{\omega_n}{n(n+2)}$;
    \item $T(t\Omega) = t^{n+2}T(\Omega)$ for all $t>0$.
  \end{enumerate}
\end{lemma}
\begin{proof}
  By exploiting the symmetry of a ball $B_r$, it is easy to verify that the function
  \[
    u_{B_r} \eqdef \frac{r^2 - |x|^2}{2n}
  \]
  solves the Poisson equation~\eqref{eq.torsion_function} for $\Omega = B_r$. From uniqueness of solutions to this PDE in this smooth setting, it must be the torsion function.

  Integrating $u_{B_r}$ in polar
  coordinates, and writing $\sigma_{n-1}=\mathcal H^{n-1}(\partial B_1)=n\omega_n$,
  \[
    T(B_r)=\int_{B_r}\frac{r^2-|x|^2}{2n}\dd x
    =\frac{\sigma_{n-1}}{2n}\int_0^r(r^2-\rho^2)\rho^{n-1}\dd \rho
    = \frac{\omega_n}{n(n+2)}\,r^{\,n+2}.
  \]

  The above formula gives the torsion of the unit ball by taking $r = 1$, but also reveals the correct scaling of the torsion functional. Now taking a general $\Omega$ open subset of $\mathbb{R}^n$, if $t>0$ and
  \[
      w_t(x)\eqdef t^2 u_\Omega\left(\frac{x}{t}\right), \text{ then }
      \begin{cases}
        - \Delta w_t = 1,& \text{ in } t \Omega,\\
        w_t = 0 ,& \text{ on } t \partial \Omega.
      \end{cases}
  \]
  Once again by uniqueness we have $u_{t\Omega}=w_t$, and the change of
  variables $x=ty$ gives
  \[
    T(t\Omega)=\int_{t\Omega}w_t\dd x=\int_\Omega t^2u_\Omega(y)\,t^n\dd y=t^{n+2}T(\Omega),
  \]
  which finishes the proof.
\end{proof}

Thanks to the scaling of the torsion functional proved above, we can assume without loss of generality that the domain is normalized to satisfy $|\Omega| = \omega_n$. We then define the mass function $m: [0, \|u\|_\infty] \to [0,\omega_n]$ via the family of super level sets of $u$, that is

 \[
  E_t\eqdef\{u>t\},\qquad m(t)\eqdef|E_t|,\qquad \text{for }0<t<\|u\|_\infty.
\]
Our goal is to study the fine properties of almost every level set $E_t$ and integrate them back to the torsion functional. A key tool is comparing the perimeter of the level sets with the perimeter of a ball of the same volume, which is given by the classical isoperimetric inequality, valid for sets of finite perimeter. Therefore, we briefly introduce this class here.
\begin{definition}\label{def.set_finite_perimeter}
    Let $E\subset\mathbb{R}^n$ be a measurable set. We say that $E$ has finite perimeter if its indicator function $\mathds{1}_{E}$ admits a distributional derivative $D\mathds{1}_{E} \in \mathscr{M}(\mathbb{R}^n;\mathbb{R}^n)$ that is a finite Radon measure. In this case, we define the perimeter of $E$ as
    \[
      \Per(E) \eqdef |D\mathds{1}_{E}|(\mathbb{R}^n).
    \]
\end{definition}

This is a very abstract definition instead of naively imposing the topological boundary $\partial E$ to have a finite $(n-1)$-dimensional Hausdorff measure. This notion is more robust and allows for a more general class of sets, which is crucial for the analysis of level sets of Sobolev functions. For sets of finite perimeter the more adapted notion is the \textit{reduced boundary} denoted by $\partial^* E$, see~\cite[Chapter~15]{maggi2012sets}, which is defined as the set of points $x \in \supp D \mathds{1}_E$ such that
\[
  \nu_E(x)\eqdef
  \lim_{r \to 0^+}
  -\frac{D \mathds{1}_E(B_r(x))}{|D \mathds{1}_E|(B_r(x))}
  \text{ exists and belongs to }
  \mathbb{S}^{n-1}.
\]
The Borel vector field $\nu_E$ is called the measure-theoretic outer unit normal to $E$ on its reduced boundary $\partial^* E$. De Giorgi's structure theorem states that the reduced boundary $\partial^* E$ is $(n-1)$-rectifiable, the following characterization of the weak gradient holds
\begin{equation}\label{eq.charac_indicator_weak_grad}
  D\mathds{1}_E = -\nu_E \mathcal{H}^{n-1}\mres\partial^* E,
  \text{ and }
  \Per(E) = \mathcal{H}^{n-1}(\partial^* E).
\end{equation}
With this technology, it is natural that the isoperimetric inequality
\[
  \Per(E) \ge n\omega_n^{1/n}|E|^{(n-1)/n}
\]
holds for sets of finite perimeter, see for instance~\cite{ambrosio2000functions,maggi2012sets}

These properties can be lifted to the analogous class of functions of bounded variation over an open $\Omega \subset \mathbb{R}^n$ and denoted by $\BV(\Omega)$. We say that $f \in L^1(\Omega)$ belongs to $\BV(\Omega)$ if its distributional derivative $Df \in \mathscr{M}(\Omega;\mathbb{R}^n)$ is a finite Radon measure. The total variation of $f$ is then defined as
\[
  \TV(f) \eqdef |Df|(\Omega),
\]
and provides the link to the perimeter of sets of finite perimeter via the co-area formula~\cite{ambrosio2000functions,EvansGariepy2015}, which states that
\[
  \TV(f) = \int_{-\infty}^{+\infty} \Per(\{f>t\}) \dd t.
\]

\subsection{The qualitative Saint-Venant inequality}\label{sec.qualitative_SV}
Coming back with this technology to the torsion function $u$, defined in~\eqref{eq.torsion_function}, we have that $u \in W_0^{1,2}(\Omega) \subset \BV(\Omega)$, and therefore the co-area formula applies to $u$. This allows us to study the level sets $E_t$ of $u$ and their perimeters $\Per(E_t)$ for almost every $t \in (0, \|u\|_\infty)$. To this end, it is useful to introduce the following quantities that will be recurrent in this work
\[
  \alpha(t)\eqdef \int_{\partial^*E_t}\frac{1}{|\nabla u|}\dHn,\quad
  \beta(t)\eqdef\int_{\partial^*E_t}|\nabla u|\dHn,\quad
  c_n\eqdef n^2\omega_n^{2/n}.
\]

\begin{lemma}\label{lemma.preliminary_alpha_beta}
The function $m$ is continuous and strictly decreasing on $(0,\|u\|_\infty)$,
and is locally absolutely continuous there. For almost every $t\in(0,\|u\|_\infty)$, the set $E_t$ has finite perimeter, its measure-theoretic (inner) unit normal is given by
\[
  \nu_{E_t}=-\frac{\nabla u}{|\nabla u|}\quad\text{on }\partial^*E_t,
\]
and
 \[
  -m'(t)=\alpha(t),\qquad \beta(t)=m(t).
\]
Moreover $\|u\|_\infty\le(2n)^{-1}$.
\end{lemma}
\begin{proof}
Extend $u$ by zero outside $\Omega$. Since $|\Omega|<\infty$,
$u\in W_0^{1,2}(\Omega)\subset W_0^{1,1}(\Omega)\subset\BV(\mathbb R^n)$, the coarea formula gives
  \[
    \int_a^b\Per(E_t)\dd t =\int_{\{a<u<b\}}|\nabla u|\dd x<\infty
\]
for $0\le a<b\le\|u\|_\infty$. Thus $E_t$ has finite perimeter for almost every $t$. Local elliptic regularity and Stampacchia's lemma imply
$|\{\nabla u=0\}|=0$: indeed, the second derivatives of $u$ vanish almost
everywhere on $\{\nabla u=0\}$, whereas $\Delta u=-1$ in $\Omega$.
The coarea formula applied with $|\nabla u|^{-1}$ therefore gives
\[
  m(a)-m(b)
  =\int_a^b\int_{\partial^*E_t}\frac{1}{|\nabla u|}\dHn\dd t.
\]
Consequently, $m$ is locally absolutely continuous and
$-m'(t)=\alpha(t)$ for almost every $t$. In addition, the coarea formula also gives
$\nu_{E_t}=-\nabla u/|\nabla u|$ on $\partial^*E_t$ for almost every $t$.

It remains to prove the identity $\beta(t) = m(t)$ for almost every $t$. It is tempting to use $\mathbf1_{E_t}$ as a test function in the Poisson equation solved by $u$. Although $\mathbf1_{E_t}$ is not in $W^{1,2}_0(\Omega)$, notice that for all $t > 0$, by the maximum principle $E_t \cap \partial \Omega = \emptyset$, therefore from classical elliptic regularity $u$ is smooth over $E_t$. 

As a result, integration by parts using the finite Radon measure $D\mathbf{1}_{E_t}$ gives
\begin{equation}
    m(t) 
    = |E_t| 
    =  -\int_{E_t} \Delta u \dd x 
    =  \int_{\Omega} \nabla u \cdot \dd D \mathbf{1}_{E_t} 
    =  \int_{\partial^* E_t} |\nabla u| \dd \mathcal{H}^{n-1}
    = \beta(t),
\end{equation}
for every $t > 0$ such that $E_t$ has finite perimeter. Notice that in the $\BV$-integration by parts above we have used the outer unit normal $\nabla u/|\nabla u| = - \nu_{E_t}$. 

%

Finally, for the $L^\infty$ bound, the Cauchy--Schwarz and isoperimetric inequalities now yield
\begin{align*}
  \alpha(t)\beta(t)
  &
  \ge\Per(E_t)^2
  \ge c_n m(t)^{2-2/n}.
\end{align*}
Thus
\[
  -m'(t)=\alpha(t)\ge c_n m(t)^{1-2/n}>0
\]
whenever $m(t)>0$. Since the torsion function is positive almost everywhere
in $\Omega$, this proves continuity and strict decrease on
$(0,\|u\|_\infty)$. Finally, using the identities relating $\alpha, \beta$ and $m$ above we have
\[
  -\frac{\dd}{\dd t}m(t)^{2/n}\ge\frac{2}{n} c_n
\]
for almost every $t$. Integrating from $\varepsilon$ to
$\|u\|_\infty$ and then letting $\varepsilon\downarrow0$ gives
\[
  \|u\|_\infty\le\frac{n}{2}\frac{\omega_n^{2/n}}{c_n}=\frac{1}{2n}.
\]
\end{proof}

This shows the interest of defining the following deficits
\begin{align}\label{eq.deficits_DH_DI}
  D_H(t)\eqdef \alpha(t)\beta(t)-\Per(E_t)^2,
  \quad
  D_I(t)\eqdef \Per(E_t)^2-c_n\,|E_t|^{2-2/n}.
\end{align}
Notice that both are non-negative quantities; while $D_H(t) \ge 0$ thanks to the Cauchy-Schwarz inequality, $D_I(t) \ge 0$ for each $t$ such that $E_t$ is a set of finite perimeter, due to the classical isotropic isoperimetric inequality. As a result, $D_I(t)$ measures the distance of the level set $E_t$ from a ball. On the other hand, from the equality cases of the Cauchy-Schwarz inequality, the deficit $D_H(t)$ gives a measure of oscillations of the gradient $|\nabla u|$ over $\partial^*E_t$.

Our goal is to relate these quantities with the deficit of non-optimality of the Saint-Venant inequality, which is defined as
\begin{equation}\label{eq.deficit_SaintVenant}
  \delta_{\SV}(\Omega)
  \eqdef T(B_1) - T(\Omega),
\end{equation}
notice that the deficit assumes this form since we have normalized the volume of $\Omega$ to be $|\Omega| = \omega_n$. 
This relation gives a direct proof of the torsion inequality and a route to
quantitative stability through quantitative forms of the inequalities that
make $D_H$ and $D_I$ nonnegative.

\begin{theorem}\label{thm.SaintVenant_qualitative}
  Let $\Omega\subset\mathbb{R}^n$ be open with $|\Omega|=\omega_n$, and let $u$
be its torsion function. Then
\[
  \delta_{\SV}(\Omega)=\int_0^{\|u_\Omega\|_\infty}
  \frac{D_H(t)+D_I(t)}{c_n\,m(t)^{1-2/n}}\dd t .
\]
In particular, since $D_H(t)$ and $D_I(t)$ are nonnegative for a.e.~$t$, this identity implies the Saint-Venant inequality.
\end{theorem}
\begin{proof}
  Let $u^\circ$ be the decreasing rearrangement of $u$, to recall is definition, first we define the one-dimensional profile
  \[
    u^\circ(s) 
    \eqdef 
    \inf 
    \left\{
        t \in \mathbb{R}: 
        |\{u > t\}| \le s
    \right\},
  \]
  and we set, with a slight abuse of notation, $u^\circ (x) \eqdef u^\circ(\omega_n|x|^n)$. Define the following quantities
\[
  I(s)\eqdef\int_0^s u^\circ(\sigma)\dd \sigma,\qquad
  G(s)\eqdef I(s)+\frac{s^{1+2/n}}{(4+2n)\,\omega_n^{2/n}} .
\]
Since $m$ is strictly decreasing and locally absolutely continuous,
$m(u^\circ(s))=s$ and the theorem on inverses of monotone absolutely
continuous functions gives, for almost every $s\in(0,\omega_n)$,
\[
  (u^\circ)'(s)=\frac{1}{m'(u^\circ(s))}=-\frac{1}{\alpha(u^\circ(s))}.
\]
Computing the derivatives of $G$, we have
\begin{align*}
  G'(s) &=I'(s)+\Bigl(1+\frac{2}{n}\Bigr)\frac{s^{2/n}}{(4+2n)\,\omega_n^{2/n}}
         = u^\circ(s)+ \frac{1}{2n \omega_n^{2/n}} s^{2/n}\\
  G''(s)&=-\frac{1}{\alpha(u^\circ(s))}+\frac{s^{2/n-1}}{c_n}.
\end{align*}

A simple integration, using the change of variables $s=m(t)$, gives
\begin{align*}
  \int_0^{\omega_n}s G''(s)\dd s
  &=\int_0^{\omega_n}\left(\frac{s^{2/n}}{c_n}-\frac{s}{\alpha(u^\circ(s))}\right)\dd s\\[4pt]
  &=\int_0^{\|u_\Omega\|_\infty}\left(\frac{m(t)^{2/n}}{c_n}
     -\frac{m(t)}{\alpha\bigl(\underbrace{u^\circ(m(t))}_{=t}\bigr)}\right)\alpha(t)\dd t\\[4pt]
  &=\int_0^{\|u_\Omega\|_\infty}\left(\frac{m(t)\,\alpha(t)}{c_n\,m(t)^{1-2/n}}-m(t)\right)\dd t\\[4pt]
 &=\int_0^{\|u_\Omega\|_\infty}\left(\left(\frac{\alpha(t)\beta(t)-\Per(E_t)^2}{c_n\,m(t)^{1-2/n}}\right)
     +\left(\frac{\Per(E_t)^2-c_n\,m(t)^{2-\frac{2}{n}}}{c_n\,m(t)^{1-2/n}}\right)\right)\dd t\\[4pt]
  &=\int_0^{\|u_\Omega\|_\infty}\frac{D_H(t)+D_I(t)}{c_n\,m(t)^{1-2/n}}\dd t .
\end{align*}
Integration by parts gives
\begin{align*}
  \int_0^{\omega_n}s\,G''(s)\dd s
  &= \frac{\omega_n}{2n} + G(0)-G(\omega_n).
\end{align*}
But $\dsty G(s)=I(s)+\frac{s^{1+2/n}}{(4+2n)\,\omega_n^{2/n}}$, and from the
preservation of mass property of decreasing rearrangements it holds that
$\dsty I(\omega_n)=\int_0^{\omega_n} u^\circ(\sigma)\dd \sigma = \int_\Omega u \dd x = T(\Omega)$. As a result, we have that
\begin{align*}
  \int_0^{\|u_\Omega\|_\infty}\frac{D_H(t)+D_I(t)}{c_n\,m(t)^{1-2/n}}\dd t
  &=\frac{\omega_n}{2n}-\left(I(\omega_n)+\frac{\omega_n}{2(2+n)}\right)
  =\frac{\omega_n}{n(n+2)} - T(\Omega)\\
  &=
  T(B_1) - T(\Omega)
  =
  \delta_{\SV}(\Omega),
\end{align*}
where we used $T(B_1) = \frac{\omega_n}{n(n+2)}$ from Lemma~\ref{lemma.torsion_scaling}. This finishes the proof.
\end{proof}

\subsection{A suboptimal quantitative result}\label{sec.subptimal_quantitative_SV}
Given Theorem ~\ref{thm.SaintVenant_qualitative}, a natural way of obtaining a quantitative stability result for the Saint-Venant inequality is to apply the sharp
quantitative isoperimetric inequality on each level set. For this inequality an effective constant is provided for instance by Figalli--Maggi--Pratelli~\cite{figalli2010mass}. Following this work, there exists a computable constant $\kappa_n>0$ such that, for every set of finite perimeter $E$ with
$|E|=\omega_n r_E^n$,
\begin{equation}\label{eq.quant_isoperimetric}
  \Per(E)-\Per(B_{r_E})
  \ge \kappa_n r_E^{n-1}\mathcal A(E)^2,
  \qquad
  \mathcal A(E)\eqdef
  \inf_{z\in\mathbb R^n}\frac{|E\Delta B_{r_E}(z)|}{|E|}.
\end{equation}
All constants below that use quantitative isoperimetry are expressed in terms
of this fixed $\kappa_n$.

However, this strategy, similarly to the argument from Fusco--Maggi--Pratelli~\cite{FuscoMaggiPratelli2009Eigenvalue}, results in a suboptimal stability exponent. As we will need it in the passage from the stability in a bounded setting to global, we give a short proof for the reader's convenience and for comparison with our argument in Section~\ref{sec.stability_SV_bounded}.

Inspired by the identity for $\delta_{\SV}(\Omega)$ from Theorem~\ref{thm.SaintVenant_qualitative}, we would like to avoid the factor $m(t)$ in the denominator. For this reason, we introduce the change of variables $t=t(\rho)$ defined as
\begin{equation}\label{eq.change_var_t_rho}
    E_\rho \eqdef \{u>t(\rho)\},
    \text{ where $t(\rho)$ is s.t. }
    |E_\rho|=|B_\rho|=\omega_n\rho^n.
\end{equation}
that gives the explicit expression $m(t(\rho))=|B_\rho|=\omega_n\rho^n$ at the cost of introducing the Jacobian
\begin{equation}\label{eq.change_var_jacobian}
    w(\rho)\eqdef -t'(\rho),\qquad \dd \mu(\rho)\eqdef w(\rho)\dd\rho .
\end{equation}
Since the map $\rho \mapsto t(\rho)$ is monotone, it is almost everywhere differentiable and hence the function $w$ well-defined. A useful identity coming from these definitions that will appear frequently is the following one
\begin{equation}\label{eq.w_alpha_P_identity}
    w(\rho)\,\alpha(t(\rho))=\Per(B_\rho),
\end{equation}
which is an easy consequence of the chain rule. Indeed, differentiating both sides of
$\omega_n \rho^n=m(t(\rho))$ gives
\begin{align*}
  \Per(B_\rho)
  =
  \frac{\dd}{\dd\rho}m(t(\rho))
  =m'(t(\rho)) t'(\rho)
  =\bigl(-\alpha(t(\rho))\bigr)\bigl(-w(\rho)\bigr)
  =w(\rho)\alpha(t(\rho)).
\end{align*}
We let $\nu_\rho$ denote the measure-theoretic outer unit normal to $E_\rho$
on its reduced boundary:
\[
  \nu_\rho\eqdef
  \nu_{E_\rho}=-\frac{\nabla u(x)}{|\nabla u|}\qquad\text{for }x\in\partial^* E_\rho.
\]

With this notation, we prove the following properties for this change of variables
\begin{proposition}\label{prop.various_estimates}
Let $t = t(\rho)$ and $w(\rho)$ be the change of variables and its Jacobian defined in~\eqref{eq.change_var_t_rho} and~\eqref{eq.change_var_jacobian} respectively. Then it holds that
\[
  0<w(\rho)\le\frac{\rho}{n}\quad\text{ a.e., and }
  \delta_{\SV}(\Omega)
  =
  \omega_n
  \int_0^1\rho^n\left[\frac{\rho}{n}-w(\rho)\right]\dd\rho.
\]
\end{proposition}
\begin{proof}
  Using the change of variables $t = t(\rho)$ defined in~\eqref{eq.change_var_jacobian} above, the identity for the torsion deficit from Theorem~\ref{thm.SaintVenant_qualitative} becomes
\begin{align*}
  \delta_{\SV}(\Omega)
  &=\int_0^{\|u\|_\infty}\frac{D_H(t)+D_I(t)}{c_n\,m(t)^{1-2/n}}\dd t
  =\int_0^1\frac{D_H(t(\rho))+D_I(t(\rho))}{c_n m(t(\rho))^{1-2/n}}
     \bigl(-t'(\rho)\bigr)\dd \rho \\[4pt]
  &=\int_0^1\frac{D_H(t(\rho))+D_I(t(\rho))}{n^2\,\omega_n\,\rho^{n-2}}\dd \mu(\rho),
\end{align*}
On the other hand, from the definition of $D_H$ and $D_I$ we recall that
\[
  0 \le D_H(t)+D_I(t)=\alpha(t)\beta(t)-c_n\,m(t)^{2-\frac{2}{n}},
\]
evaluated at $t=t(\rho)$ we have
\begin{align*}
  D_H(t(\rho))+D_I(t(\rho))
  &=\omega_n\rho^n\,\alpha(t(\rho))-n^2\omega_n^{2/n}\bigl(\omega_n\rho^n\bigr)^{2-\frac{2}{n}}\\[4pt]
  &=\omega_n\rho^n\Bigl(\alpha(t(\rho))-n^2\omega_n\rho^{n-2}\Bigr).
\end{align*}
In addition, at almost every radius under consideration, $\alpha(t(\rho))<\infty$ and
$\Per(B_\rho)>0$; hence \eqref{eq.w_alpha_P_identity} gives $w(\rho)>0$.

Now using that $w(\rho)\,\alpha(t(\rho))=\Per(B_\rho)$, we get
\begin{align*}
  0&\le\omega_n\rho^n\left(\frac{n\,\omega_n\,\rho^{n-1}}{w(\rho)}-n^2\omega_n\rho^{n-2}\right)
   =n\,\omega_n^2\,\rho^{2n-2}\left(\frac{\rho}{w(\rho)}-n\right)
\end{align*}
as a result we get $0<w(\rho)\le\dfrac{\rho}{n}$ for almost every $\rho \in (0,1)$.
Using this in the formula for $\delta_{\SV}(\Omega)$:
\begin{align*}
  \delta_{\SV}(\Omega)
  &=\int_0^1\frac{\omega_n\rho^n\bigl(\alpha(t(\rho))-n^2\omega_n\rho^{n-2}\bigr)w(\rho)}{n^2\,\omega_n\,\rho^{n-2}} \dd \rho\\[4pt]
  &=\int_0^1\omega_n\rho^n\left[\frac{\overbrace{w(\rho)\,\alpha(t(\rho))}^{=\Per(B_\rho)=n\omega_n\rho^{n-1}}
     -\,n^2\omega_n\rho^{n-2}\,w(\rho)}{n^2\,\omega_n\,\rho^{n-2}}\right]\dd\rho\\[4pt]
  &=\int_0^1\omega_n\rho^n\left[\frac{n\,\omega_n\,\rho^{n-1}-n^2\omega_n\rho^{n-2}w(\rho)}{n^2\,\omega_n\,\rho^{n-2}}\right]\dd\rho\\[4pt]
  &=\int_0^1\omega_n\rho^n\left[\frac{\rho}{n}-w(\rho)\right]\dd\rho.
\end{align*}
\end{proof}
This machinery greatly helps in proving a provisional stability result and will be useful later in the proof of the version with sharp exponent.

\begin{theorem}[Suboptimal stability]\label{thm:suboptimal_stabilitySV}
There is a computable constant $C_3(n)>0$ such that every open set
$\Omega\subset\mathbb R^n$ with $|\Omega|=\omega_n$ satisfies
\begin{equation}
   \mathcal{A}(\Omega)^3\le C_3(n)\delta_{\SV}(\Omega).
\end{equation}
\end{theorem}
\begin{proof}
If $\mathcal A(\Omega)=0$, there is nothing to prove. Assume henceforth that
$\mathcal A(\Omega)>0$. Using the isoperimetric inequality~\eqref{eq.quant_isoperimetric}, we obtain for almost every $\rho$ that
\begin{equation}\label{eq:DI-lower}
 D_I(t(\rho))
 \ge c(n)\rho^{2n-2}\mathcal{A}(E_\rho)^2.
\end{equation}
Indeed we can rewrite the deficit as
  \[
  D_I(t(\rho))
 =\bigl(\Per(E_\rho)-\Per(B_\rho)\bigr)
  \bigl(\Per(E_\rho)+\Per(B_\rho)\bigr),
  \]
apply the quantitative isoperimetric inequality~\eqref{eq.quant_isoperimetric} to the first factor and bound the sum of perimeters from below by $2n\omega_n^{1/n}\rho^{n-1}$, since once again the isoperimetric inequality gives $\Per(E_\rho) \ge \Per(B_\rho)$. 

Let $\rho_A=(1-\mathcal{A}(\Omega)/4)^{1/n}$, using the triangular inequality for any
$\rho\in[\rho_A,1]$ yields that
\[
 \mathcal{A}(\Omega)\le \mathcal{A}(E_\rho)+2(1-\rho^n)
   \le \mathcal{A}(E_\rho)+2(1-\rho_A^n)
   =\mathcal{A}(E_\rho)+\frac{\mathcal{A}(\Omega)}{2}.
\]
Consequently, $\mathcal{A}(E_\rho)\ge \mathcal{A}(\Omega)/2$ throughout this interval.

Using \eqref{eq:DI-lower}, and bounding $D_H(t(\rho)) \ge 0$ in the exact identity for $\delta_{\rm SV}$ from Theorem~\ref{thm.SaintVenant_qualitative} gives
\begin{equation}\label{eq:seed-collar-short}
 \delta_{\SV}(\Omega)\ge c(n)\mathcal{A}(\Omega)^2
 \int_{\rho_A}^1\rho^n w (\rho)\dd \rho .
\end{equation}
Next, adding and subtracting
the ball profile $\rho/n$, we obtain
\begin{align*}
 \int_{\rho_A}^1\rho^n w (\rho)\dd \rho
 &=\int_{\rho_A}^1\frac{\rho^{n+1}}{n}\dd \rho
   -\int_{\rho_A}^1\rho^n
      \left(\frac{\rho}{n}- w (\rho)\right)\dd \rho\\
 &=\frac{1-\rho_A^{n+2}}{n(n+2)}
   -\int_{\rho_A}^1\rho^n
      \left(\frac{\rho}{n}- w (\rho)\right)\dd\rho.
\end{align*}
The integrand in the last term is nonnegative thanks to Proposition~\ref{prop.various_estimates}, which also implies that
\[
 \int_{\rho_A}^1\rho^n
      \left(\frac{\rho}{n}- w (\rho)\right)\dd\rho
 \le \int_0^1\rho^n
      \left(\frac{\rho}{n}- w (\rho)\right)\dd\rho
 =\frac{1}{\omega_n}\delta_{\SV}(\Omega).
\]
Moreover, since $\rho_A < 1$
\[
 1-\rho_A^{n+2}
 \ge
 1-\rho_A^{n}
 =\frac{\mathcal{A}(\Omega)}{4}.
\]
Hence
\begin{equation}\label{eq:seed-collar-lower}
 \int_{\rho_A}^1\rho^n w (\rho)\dd \rho
 \ge c(n)\mathcal{A}(\Omega)-C(n)\delta_{\SV}(\Omega).
\end{equation}
Combining \eqref{eq:seed-collar-short} and
\eqref{eq:seed-collar-lower}, we find
\[
 \delta_{\SV}(\Omega)
 \ge c(n)\mathcal{A}(\Omega)^3-C(n)\mathcal{A}(\Omega)^2\delta_{\SV}(\Omega),
\]
or equivalently
\[
 \bigl(1+C(n)\mathcal{A}(\Omega)^2\bigr)\delta_{\SV}(\Omega)\ge c(n)\mathcal{A}(\Omega)^3.
\]
Since every Fraenkel asymmetry satisfies $0\le \mathcal{A}(\Omega)<2$, the factor on
the left is bounded above by a dimensional constant. Thus
$\mathcal{A}(\Omega)^3\le C_3(n)\delta_{\SV}(\Omega)$. The constant $C_3(n)$ is computable because it is obtained from $n$, $\omega_n$, and the fixed computable constant $\kappa_n$ by finitely many
arithmetic operations.
\end{proof}

\section{Sharp quantitative Saint-Venant inequality in a bounded domain}\label{sec.stability_SV_bounded}

  The goal of this section is to first prove quantitative stability on a bounded domain, therefore we fix $R>2$ and an open set $\Omega\subset B_R(0)$ with $|\Omega|=\omega_n$. Our goal is to relate $D_H(t(\rho))$ and $D_I(t(\rho))$ to the derivative of the distance between $E_\rho$ and the family of balls of radius $\rho$. Define

 \begin{equation}\label{eq.moving_ball_distance}
  q(\rho)\eqdef\inf_{z\in\mathbb{R}^n}\bigl|E_\rho\,\Delta\,B_\rho(z)\bigr|,
\end{equation}
which we also call the \textit{moving ball distance}. A few easy properties follow directly from the definition:

\begin{lemma}\label{lemma.moving_ball_properties}
  The following properties of the function $(0,1) \ni \rho \mapsto q(\rho)$ hold
  \begin{enumerate}
    \item For all $0\le\rho\le 1$ we have $q(\rho)\le 2\omega_n$;
    \item The infimum in the definition is attained.
    \item The mapping $\rho \mapsto q(\rho)$ is $2n\omega_n$--Lipschitz.
  \end{enumerate}
\end{lemma}
\begin{proof}
  Item (1) follows directly from the fact that both sets have volume at most $\omega_n$. For item (2), notice first that for any center $z$ such that
  \[
    |E_\rho \Delta B_\rho(z)| \ge 2 \omega_n\rho^n, 
  \]
  we must have that $E_\rho$ and $B_\rho(z)$ are disjoint. Hence taking any center such that $E_\rho \cap B_\rho(z) \neq \emptyset$ will reduce this value. As a result, we may consider only $z \in \overline{B_{R + \rho}}$. Compactness and continuity with respect to the center give a minimizer.

  For item (3), we need only a simple argument using the triangle inequality. Consider $0 < \sigma < \rho < 1$ and let $z_i \in \mathbb{R}^n$ be optimal in the definition of $q(i)$ for $i \in \{\sigma, \rho\}$.
  Then, using the elementary identity that $|A \Delta B| = |B| - |A|$ whenever $A \subset B$, we have that
  \begin{align*}
    q(\rho) \le
    |E_\rho \Delta B_\rho(z_\sigma)|
    &\le
    |E_\rho \Delta E_\sigma|
    +
    |E_\sigma \Delta B_\sigma(z_\sigma)|
    +
    |B_\sigma(z_\sigma) \Delta B_\rho(z_\sigma)| \\
    &=
    q(\sigma) + 2 \omega_n(\rho^n - \sigma^n)
    \le
     q(\sigma) + 2 n \omega_n(\rho - \sigma).
  \end{align*}
  Changing the roles of $\rho$ and $\sigma$, the result follows.
\end{proof}

This last property implies of course that $\rho \mapsto q(\rho)$ is differentiable almost everywhere, and our goal for now is to estimate its derivatives. In these estimates the following quantities will naturally appear
\begin{equation}\label{eq.velocity_normals}
    V_\rho\eqdef\frac{w(\rho)}{|\nabla u|},\qquad
    H_{z,\rho}(x)\eqdef\frac{(x-z)\cdot\nu_\rho(x)}{\rho},\qquad
    \overline{V}_\rho\eqdef\fint_{\partial^*E_\rho}V_\rho\ =\ \frac{\Per(B_\rho)}{\Per(E_\rho)},
\end{equation}
where the last equation is due to
\[
  \int_{\partial^*E_\rho}V_\rho\dHn=w(\rho)\int_{\partial^*E_\rho}\frac{1}{|\nabla u|}\dHn
  =w(\rho)\,\alpha(t(\rho))=\Per(B_\rho).
\]

With this notation, we can estimate the derivative of the moving ball distance from above.
\begin{lemma}\label{lemma.estimate_moving_ball_derivative}
  Assume that $\Omega$ is an open set with $\partial \Omega$ smooth.
  For each radius $\rho \in (0,1)$, let $z_\rho$ be a minimizer in the definition of $q(\rho)$. For almost every radius $\rho$ we have the following bound 
  \begin{equation}\label{eq.estimate_dev_q}
    q'(\rho) \le \frac{n}{\rho}\,q(\rho)
    +\int_{\partial^*E_\rho}\bigl|V_\rho-H_{z_\rho,\rho}\bigr|\dHn.
  \end{equation}
\end{lemma}

In the proof of this result, we use the following variable-width coarea formula.
\begin{lemma}\label{lemma.variable_width_coarea}
Let $\Omega$ be an open set with smooth boundary and $\xi \le \zeta$ be continuous functions over $\overline{\Omega}$. Then, for almost every $s \in (0, \norm{u}_\infty)$ we have that 
\[
    \lim_{h \to 0^+} 
    \frac{1}{h}
    \left|
        \left\{
            h \xi < u - s < h \zeta
        \right\}
    \right|
    = 
    \int_{\partial^*\{u > s\}} 
    \frac{\zeta - \xi}{|\nabla u|} \dd \mathcal{H}^{n-1},
\]
where the exceptional set depends only on $\xi$ and $\zeta$.
\end{lemma}
\begin{proof}
    As in the proof of Lemma~\ref{lemma.preliminary_alpha_beta}, by the maximum principle, the set $\{\nabla u = 0\}$ is negligible, so we can apply the coarea formula. First we prove this result for $\xi \le \zeta$ being simple functions of the form 
    \[
        \xi = \sum_{i = 1}^k p_i \mathbf{1}_{A_i}, \quad 
        \zeta = \sum_{i = 1}^k q_i \mathbf{1}_{A_i}, 
        \text{ with } 
        p_i \le q_i, 
    \]
    where the sets $A_i$ are pair-wise disjoint. For each $i$, the coarea formula gives 
    \[
        \left|
        \left\{
            h p_i < u - s < h q_i
        \right\} \cap A_i
        \right|
        = 
        \int_{s + h p_i}^{s + h q_i}
        \int_{\partial^*\{u > t\}} 
        \frac{\mathbf{1}_{A_i}}{|\nabla u|} 
        \dd\mathcal{H}^{n-1}\dd t.
    \]
    As a result, for every Lebesgue point of the one-variable function $t \mapsto \int_{\partial^*\{u > t\}} \frac{\mathbf{1}_{A_i}}{|\nabla u|} \dd\mathcal{H}^{n-1}$, we have that 
    \[
        \lim_{h \to 0^+} 
        \frac{1}{h}
        \left|
        \left\{
            h p_i < u - s < h q_i
        \right\} \cap A_i
        \right|
        = 
        (q_i - p_i)
        \int_{\partial^*\{u > s\}} 
        \frac{\mathbf{1}_{A_i}}{|\nabla u|} 
        \dd\mathcal{H}^{n-1}.
    \]
    Since the sets are pair-wise disjoint, by linearity the conclusion follows for $\xi, \zeta$ as above. 

    For general $\xi, \zeta$, we approximate the nonnegative difference $\zeta - \xi$ with simple functions as above. First, we consider an outer approximation with simple functions $(\xi_k, \zeta_k)$ such that $\zeta \le \zeta_k$ and $\xi_k \le \xi$. Then we get that 
    \begin{align*}
        \limsup_{h \to 0^+} 
        \frac{1}{h}
        \left|
        \left\{
            h\xi < u - s < h \zeta
        \right\}
        \right|
        &\le 
        \lim_{h \to 0^+} 
        \frac{1}{h}
        \left|
        \left\{
            h\xi_k < u - s < h \zeta_k
        \right\}
        \right|\\
        &=
        \int_{\partial^*\{u > s\}} 
        \frac{\zeta_k - \xi_k}{|\nabla u|} 
        \dd\mathcal{H}^{n-1}
        \xrightarrow[k \to \infty]{} 
        \int_{\partial^*\{u > s\}} 
        \frac{\zeta - \xi}{|\nabla u|} 
        \dd\mathcal{H}^{n-1}. 
    \end{align*}
    For each $k$ the limit above holds in a set of full measure depending on the approximation $(\zeta_k, \xi_k)$, taking the union of all exceptional sets still gives an exceptional set of measure zero. Considering now an inner approximation by simple functions such that $\zeta \ge \zeta_k$ and $\xi_k \ge \xi$, the result follows.
\end{proof}

\begin{proof}[Proof of Lemma~\ref{lemma.estimate_moving_ball_derivative}]
Since $q$ is Lipschitz, it is almost everywhere differentiable so whenever its derivative exists, it coincides with
\begin{align*}
  D^+q(\rho)\eqdef \limsup_{h\to0^+}\frac{q(\rho+h)-q(\rho)}{h}
  \le\limsup_{h\to0^+}
  \frac{\bigl|E_{\rho+h}\Delta B_{\rho+h}(z_\rho)\bigr|-\bigl|E_\rho\Delta B_\rho(z_\rho)\bigr|}{h},
\end{align*}
where the upper bound is obtained by using $z_\rho$ in the infimum defining $q(\rho + h)$. 

To compare $q(\rho+h)$ to $q(\rho)$, we consider $\rho$ such that $r\mapsto t(r)$ is differentiable, so that $t(\rho+h)=t(\rho)-h\,w(\rho)+o(h)$, and we consider the homothetic map
\[
  T_h(x)\eqdef z_\rho+\Bigl(\frac{\rho+h}{\rho}\Bigr)(x-z_\rho),
\]
that is $T_h$ maps $B_\rho(z_\rho)$ onto $B_{\rho+h}(z_\rho)$. Hence
\begin{align*}
  q(\rho+h)
  &\le\bigl|E_{\rho+h}\Delta T_h(E_\rho)\bigr|
   +\bigl|T_h(E_\rho)\Delta T_h(B_\rho(z_\rho))\bigr|\\
  &=\bigl|E_{\rho+h}\Delta T_h(E_\rho)\bigr|
   +\Bigl(1+\frac{h}{\rho}\Bigr)^{n}q(\rho)
\end{align*}
So that 
\[
  q(\rho+h) - q(\rho) \le \bigl|E_{\rho+h}\Delta T_h(E_\rho)\bigr|
   +h\frac{n}{\rho}q(\rho) + o(h), 
\]
and to conclude it suffices to prove that
\[
  \limsup_{h\to0^+}\frac{\bigl|E_{\rho+h} \Delta T_h(E_\rho)\bigr|}{h}
  \ \le\ \int_{\partial^*E_\rho}\bigl|V_\rho-H_{z_\rho,\rho}\bigr|\dHn .
\]
Recalling that $t(\rho+h)=t(\rho)-h\,w(\rho)+o(h)$, one can check that
\begin{align*}
  T_h^{-1}(x)
  &=
  x-\frac{h}{\rho}(x-z_\rho)+O(h^2)\\  
  u\bigl(T_h^{-1}(x)\bigr)
  &=
  u(x)-\frac{h}{\rho}\nabla u(x)\cdot(x-z_\rho)+o(h). 
\end{align*}
In addition, since we have assumed $\partial \Omega$ to be smooth, from classical elliptic regularity theory it holds that $u \in C^1(\bar \Omega)$, and therefore the remainders are uniform among $x \in E_{\rho + \bar h}$. As a result, for any radius $0 < \rho < 1$, there is some $\bar h$ such that $\rho + \bar h < 1$ and for any $h < \bar h$ we have 
\[
  E_{\rho + h} \subset E_{\rho + \bar h}, \quad 
  T_h(E_\rho) \subset E_{\rho + \bar h}, 
  \text{ and }
  \quad 
  \overline{E_{\rho + \bar h}} \Subset \Omega,
\]
where the second inclusion follows from the fact that $T_h$ converges locally uniformly to the identity. Hence, for all sufficiently small $h>0$ we have
\[
 \left|t(\rho+h)-t(\rho)+hw(\rho)\right|\le\varepsilon h
\]
and
\[
 \sup_{x\in \overline{E_{\rho + \bar h}}}\left|
 u(T_h^{-1}(x))-u(x)
 +h\nabla u(x)\cdot\frac{x-z_\rho}{\rho}
 \right|\le\varepsilon h.
\]
As a result, in order to apply Lemma~\ref{lemma.variable_width_coarea}, we define the following bounded continuous functions
\[
 \xi_\rho(x)\eqdef
 \min\left\{-w(\rho),\nabla u(x)\cdot\frac{x-z_\rho}{\rho}\right\},
 \quad
 \zeta_\rho(x)\eqdef
 \max\left\{-w(\rho),\nabla u(x)\cdot\frac{x-z_\rho}{\rho}\right\}.
\]
With this notation, the two uniform estimates above imply
\begin{equation}\label{eq.compact_slab_inclusion}
 (E_{\rho+h}\Delta T_h(E_\rho))
 \subset
 \left\{x\in \Omega:
 h(\xi_\rho(x)-2\varepsilon)<u(x)-t(\rho)
 <h(\zeta_\rho(x)+2\varepsilon)\right\}.
\end{equation}

Formally Lemma~\ref{lemma.variable_width_coarea}, applied at $s=t(\rho)$, would now give the required limit of the measure of the right-hand side. However, since the exceptional set from Lemma~\ref{lemma.variable_width_coarea} depends on Borel functions, whereas $\xi_\rho$ and $\zeta_\rho$ depend on $\rho$, collecting these exceptional sets for all $\rho$ might lead to an union with positive measure. We avoid this by considering only a countable family of Borel functions
\[
 \xi_{a,b,c,z} \eqdef 
 \min\{a,b\nabla u(x)\cdot(x-z)\}-c,
 \quad
 \zeta_{a,b,c,z} \eqdef 
 \max\{a,b\nabla u(x)\cdot(x-z)\}+c,
\]
where $a,b,c\in\mathbb Q$, $b,c>0$, and $z\in\mathbb Q^n$. This family is countable, so Lemma~\ref{lemma.variable_width_coarea} holds for all its members outside one null set of levels. This null set also produces a null set of radii $\rho$. Indeed, for every integer $j\ge2$, we known from Lemma~\ref{lemma.preliminary_alpha_beta} that the map
\[
 s\longmapsto\left(\frac{m(s)}{\omega_n}\right)^{1/n}
\]
is absolutely continuous on the levels corresponding to $\rho>1/j$, and therefore maps null sets of levels to null sets of radii. Taking the countable union over $j$ proves that the exceptional radii form a null set in $(0,1)$. In addition, we consider also the null set of points such that $\alpha(t(\rho)) = +\infty$, since $\delta_{\rm SV}(\Omega) < +\infty$, the characterization from Theorem~\ref{thm.SaintVenant_qualitative} implies that the set of such $\rho$'s have measure zero. Hereafter, we assume that $\rho$ is outside this null set. 

Next, approximate the parameters $-w(\rho)$, $1/\rho$, $\varepsilon$ and $z_\rho$ by rational parameters $(a,b,c,z)$ in the family above. Since $\nabla u$ is continuous on $\overline{\Omega}$, the approximation is uniform. Choosing $c$ slightly larger than the approximation error and $2\varepsilon$, we can assume that 
\[
  \begin{aligned}
    \left\{
          h(\xi_\rho(x)-2\varepsilon)<u(x)-t(\rho)
          <h(\zeta_\rho(x)+2\varepsilon)
    \right\}\\ 
    \subset
    \left\{
            h\xi_{a,b,c,z} 
            < u(x)-t(\rho) < 
      h\zeta_{a,b,c,z}\right\}
  \end{aligned},
\]
while its width converges uniformly to $\zeta_\rho-\xi_\rho+4\varepsilon$. Lemma~\ref{lemma.variable_width_coarea} therefore yields
\begin{align*}
 \limsup_{h\to0^+}
 \frac{|(E_{\rho+h}\Delta T_h(E_\rho))|}{h}
 &\le
 \int_{\partial^*E_\rho}
 \frac{\zeta_\rho-\xi_\rho}{{|\nabla u|}}\dHn
 + 
 4\varepsilon\alpha(t(\rho)).
\end{align*}
Since we have taken $\rho$ such that $\alpha(t(\rho)) < +\infty$, letting $\varepsilon\to0$ and using
\[
 \zeta_\rho-\xi_\rho
 =\left|w(\rho)+\nabla u\cdot\frac{x-z_\rho}{\rho}\right|, 
 \quad 
 \nu_\rho = - \frac{\nabla u}{|\nabla u|} \text{ on } \partial^*E_\rho
\]
gives
\begin{align*}
 \limsup_{h\to0^+}
 \frac{|(E_{\rho+h}\Delta T_h(E_\rho))|}{h}
 &\le
 \int_{\partial^*E_\rho}
 \frac{\left|w(\rho)+\nabla u\cdot\frac{x-z_\rho}{\rho}\right|}
 {|\nabla u|}\dHn\\
 &=\int_{\partial^* E_\rho}
 |V_\rho-H_{z_\rho,\rho}|\dHn.
\end{align*}
This completes the proof.
\end{proof}

\begin{remark}\label{remark.no_smoothness_need}
  The assumption of $\partial \Omega$ being smooth is not necessary, but it makes the proof simpler, as we would in the former case need an argument to show that the volume of the symmetric differences divided by $h$ does not concentrate at the boundary. We make the trade-off of simplifying this proof and using a simple approximation argument in Theorem~\ref{thm:boundedSV}.
\end{remark}

In order to bridge the gap between the estimate of $q'$ and the quantities $D_H,D_I$, our major tool is going to be the quantitative isoperimetric inequality. This is very natural to expect as the isoperimetric deficit controls the Fraenkel asymmetry between a given set and the euclidean ball~\cite{figalli2010mass}. That is the natural proof strategy used in the proof of Theorem~\ref{thm:suboptimal_stabilitySV}. However, this does not help in dealing with the term $H_{z_\rho, \rho}$ defined in~\eqref{eq.velocity_normals}. In fact, this term resembles a distance of $\nu_\rho$, the normal vector field to $\partial^* E_\rho$, to the constant vector field at $\partial B_\rho$ defining its normal. Taking this into account, one is naturally induced to consider its relationship to the reinforced isoperimetric inequality proposed by Fusco--Julin in~\cite{fusco2014strong}. 

Curiously, the original proof of Fusco--Julin~\cite{fusco2014strong} also uses the Cicalese--Leonardi selection principle, hence it does not provide a computable constant, as the sharp stability of the Faber-Krahn inequality of Brasco--De Philippis--Velichkov~\cite{brasco2015faber}. Recently, however, Hensel and Laux~\cite{hensel2026quantitative} have proposed a new proof using a calibration argument, which, although in the global case their argument does not provide a computable constant, since it also uses compactness, they provide a computable constant whenever the isoperimetric deficit is sufficiently small. Fortunately, this is enough for our purposes, and will be recorded below. 

 \begin{proposition}
 \label{prop.FJ-strong}
For every   $n\ge2$ and $R_0\ge2$ there are computable constants
$\varepsilon_{\rm FJ}(n,R_0)>0$ and $C_{\rm FJ}(n,R_0)>0$ such that every  set
of finite perimeter   $F\subset B_{R_0}$ satisfying
\[
 |F|=\omega_n,\qquad \int_Fx\dd x=0,
 \qquad \Per(F)-\Per(B_1)\le\varepsilon_{\rm FJ}(n,R_0)
\]
obeys
\begin{equation}\label{eq.FJ-exact}
 |F\Delta B_1|^2
 +\int_{\partial^*F}
 \left|\nu_F-\frac{x}{|x|}\right|^2\dHn
 \le C_{\rm FJ}(n,R_0)\bigl(\Per(F)-\Per(B_1)\bigr).
\end{equation}
\end{proposition}

\begin{proof}
We use the perturbative calibration theorem of Hensel--Laux
\cite[Theorem~2]{hensel2026quantitative}, with the sign of the calibration
reversed to match our outer-normal convention. We first recall that its
smallness threshold is computable, as the perturbative proof contains no
compactness argument, since it only uses a fixed finite atlas of the sphere, the Euclidean relative isoperimetric and Sobolev inequalities, convolution with a fixed smooth kernel, and the eigenvalues $k(k+n-2)$ of the spherical
Laplacian. 

Set $\Delta(F)=\Per(F)-\Per(B_1)$. By
\eqref{eq.quant_isoperimetric}, there is $y\in\mathbb R^n$  such that
  \[
 |F\Delta B_1(y)|\le C(n)\Delta(F)^{1/2},
\]
  with computable $C(n)$. Reduce the final threshold so that the right-hand side
is less than $\omega_n$. Then $B_1(y)$ intersects $F\subset B_{R_0}$, and
$|y|\le R_0+1$. Since the barycenter of $F$ is zero,
\[
 \omega_n|y|
 =\left|\int_{B_1(y)}x\dd x-\int_Fx\dd x\right|
 \le (R_0+2)|F\Delta B_1(y)|.
\]
The translation estimate for balls now gives
\begin{equation}\label{eq.barycenter_ball}
 |F\Delta B_1|\le C(n,R_0)\Delta(F)^{1/2}.
\end{equation}

  Let $\xi$ be  the   Hensel--Laux calibration and set
\[
 \mathcal E(F)\eqdef\int_{\partial^*F}(1-\xi\cdot\nu_F)\dHn.
\]
Since $\mathcal E(B_1)=0$, Gauss--Green and
\eqref{eq.barycenter_ball} yield
\[
 \mathcal E(F)
 \le\Delta(F)+\|\ddiv\xi\|_\infty|F\Delta B_1|
 \le\Delta(F)+C(n,R_0)\Delta(F)^{1/2}.
\]
We can therefore choose a computable
$\varepsilon_{\rm FJ}(n,R_0)>0$ so that
$\mathcal E(F)\le\varepsilon_{\rm cal}(n,R_0)$. Taking $\gamma=1/2$ in the
perturbative calibration theorem gives
\[
 \mathcal E(F)\le4\Delta(F).
\]
Writing $\xi(x)=f(|x|)x/|x|$, where $0\le f\le1$, one has pointwise
\[
 \left|\nu_F-\frac{x}{|x|}\right|^2
 \le4(1-\xi\cdot\nu_F).
\]
Together with \eqref{eq.barycenter_ball}, this proves
\eqref{eq.FJ-exact}. Every threshold and constant used above is obtained from
computable quantities by finite arithmetic operations and finite searches.
\end{proof}

To simplify our estimates, further decompose inequality~\eqref{eq.estimate_dev_q} into two parts via a simple triangular inequality, this gives the following three terms
\begin{equation}\label{eq.estimate_dev_q_3terms}
  q'(\rho) \le \frac{n}{\rho}q(\rho)
  +\int_{\partial^*E_\rho}\bigl|V_\rho-\overline{V}_\rho\bigr|\dHn
  +\int_{\partial^*E_\rho}\bigl|\overline{V}_\rho-H_{z_\rho,\rho}(x)\bigr|\dHn.
\end{equation}
In the sequel we control these terms by the deficits
$D_H(t(\rho))$ and $D_I(t(\rho))$ defined in~\eqref{eq.deficits_DH_DI}. As mentioned above, the main ingredient is the strong quantitative isoperimetric inequality of Fusco-Julin. However, for each level set $E_\rho$ Proposition~\ref{prop.FJ-strong} provides a center $\bar z_\rho$ that might be different from the center $z_\rho$ that is optimal for the definition of $q(\rho)$. As a result, in a second moment, we verify that the distance between them is also controlled by these deficits, provided that these are not large.

\begin{proposition}\label{prop.estimates_q'_DH_DI}
The following estimates hold for almost every $\rho\in(0,1)$:
\begin{itemize}
\item[(1)]
  \[
     \left(\int_{\partial^*E_\rho}
    |V_\rho-\overline V_\rho|\dHn\right)^2
    \le D_H(t(\rho)).
 \]
 \item[(2)] With the computable constant $\kappa_n$ from
\eqref{eq.quant_isoperimetric},
\[
 q(\rho)^2
 \le\frac{\omega_n}{2n\kappa_n}\rho^2D_I(t(\rho)).
\]
\item[(3)] For every $\rho_*\in(0,1)$ there are computable constants
$\eta_I=\eta_I(n,R,\rho_*)>0$ and $C=C(n,R,\rho_*)>0$ such that, for almost
every $\rho\in[\rho_*,1]$ satisfying
$D_I(t(\rho))\le\eta_I$, there is a center $\bar z_\rho$ for which
\[
 |E_\rho\Delta B_\rho(\bar z_\rho)|^2
 +\left(\int_{\partial^*E_\rho}
 |H_{\bar z_\rho,\rho}-\overline V_\rho|\dHn\right)^2
 \le C D_I(t(\rho)).
\]
If  $z_\rho$ is optimal for $q(\rho)$, then also
\[
 |z_\rho-\bar z_\rho|\le C D_I(t(\rho))^{1/2},
 \qquad
 \left(\int_{\partial^*E_\rho}
 |H_{z_\rho,\rho}-\overline V_\rho|\dHn\right)^2
 \le C D_I(t(\rho)).
\]
 \end{itemize}
\end{proposition}

\begin{proof}
We prove the three estimates at radii for which the preceding identities hold and $E_\rho$ has finite perimeter. Starting with item (1), recall the definition of $V_\rho$ and $\overline{V}_\rho$ in~\eqref{eq.velocity_normals}
\[
  V_\rho\eqdef\frac{w(\rho)}{|\nabla u|},\qquad
  \overline{V}_\rho\eqdef\fint_{\partial^*E_\rho}V_\rho\dHn=\frac{\Per(B_\rho)}{\Per(E_\rho)},
\]
as well as the deficit $D_H(t(\rho))=\alpha(t(\rho))\,\beta(t(\rho))-\Per(E_\rho)^2$. By the Cauchy-Schwarz inequality, it holds that
\begin{align*}
  \left(\int_{\partial^*E_\rho}
  \bigl|V_\rho-\overline{V}_\rho\bigr|\dHn\right)^{2}
  \le
  \int_{\partial^*E_\rho}V_\rho\dHn
  \int_{\partial^*E_\rho}\frac{\bigl(V_\rho-\overline{V}_\rho\bigr)^2}{V_\rho}\dHn
\end{align*}
The first term can be computed as
\[
  \int_{\partial^*E_\rho}V_\rho\dHn = \Per(E_\rho) \bar V_\rho = \Per(B_\rho).
\]
As for the second, we develop the squares to obtain
  \begin{align*}
  \int_{\partial^*E_\rho}\frac{\bigl(V_\rho-\overline{V}_\rho\bigr)^2}{V_\rho}\dHn
  = &
  w(\rho)\int_{\partial^* E_\rho} \frac{1}{|\nabla u|} \dHn\\
  & 
  -2\bar V_\rho \Per(E_\rho)
  +
  \frac{\bar V_\rho^2}{w(\rho)} \int_{\partial^* E_\rho}|\nabla u| \dHn
\end{align*}
Hence, recalling the definitions of $\alpha(t), \beta(t)$ and that $w(\rho)\alpha(t(\rho))=\Per(B_\rho)$ we get
\begin{align*}
    \int_{\partial^*E_\rho}\frac{\bigl(V_\rho-\overline{V}_\rho\bigr)^2}{V_\rho}\dHn
    &=\frac{\Per(B_\rho)}{\Per(E_\rho)^2} \alpha(t(\rho))\,\beta(t(\rho))-\Per(B_\rho)
    =\frac{\Per(B_\rho)}{\Per(E_\rho)^2} D_H(t(\rho)).
\end{align*}
Multiplying both terms, and recalling that $\Per(B_\rho) \le \Per(E_\rho)$
by the isoperimetric inequality, since these sets have the same volume by construction, item (1) follows.

For item~(2), the quantitative isoperimetric inequality gives
\begin{align*}
 D_I(t(\rho))
 &\ge 2n\omega_n\kappa_n\rho^{2n-2}\mathcal A(E_\rho)^2
 =\frac{2n\kappa_n}{\omega_n\rho^2}q(\rho)^2,
\end{align*}
which is the asserted estimate.

We now prove item~(3). Fix $\rho_*\in(0,1)$ and
$\rho\in[\rho_*,1]$. Let
\[
 \bar z_\rho\eqdef\frac{1}{|E_\rho|}\int_{E_\rho}x\dd x
\]
be the barycenter of $E_\rho$, and define
$F_\rho=\rho^{-1}(E_\rho-\bar z_\rho)$. Since
$E_\rho\subset B_R$, one has $|\bar z_\rho|\le R$ and
$F_\rho\subset B_{2R/\rho_*}$. Moreover,
\begin{align}\label{eq.normalized_perimeter_deficit}
 \Per(F_\rho)-\Per(B_1)
 &=\frac{\Per(E_\rho)-\Per(B_\rho)}{\rho^{n-1}}
  \le\frac{D_I(t(\rho))}{2n\omega_n\rho^{2n-2}}
 \le\frac{D_I(t(\rho))}{2n\omega_n\rho_*^{2n-2}}.
\end{align}
Choose the computable number $\eta_I(n,R,\rho_*)\le1$ so that the last
quantity is at most
$\varepsilon_{\rm FJ}(n,\max\{2,2R/\rho_*\})$ whenever
$D_I(t(\rho))\le\eta_I$. Proposition~\ref{prop.FJ-strong}, followed by
scaling, then gives
\begin{align}
 |E_\rho\Delta B_\rho(\bar z_\rho)|^2
 &\le C(n,R,\rho_*)D_I(t(\rho)),
 \label{eq.local_FJ_volume}\\
 \int_{\partial^*E_\rho}
 \left|\nu_\rho-\frac{x-\bar z_\rho}{|x-\bar z_\rho|}\right|^2\dHn
 &\le C(n,R,\rho_*)D_I(t(\rho)).
 \label{eq.local_FJ_normal}
\end{align}

We first estimate the second quantity in item~(3). Recall that
$
 \overline V_\rho=\frac{\Per(B_\rho)}{\Per(E_\rho)}.
$
The triangle inequality gives
\begin{align*}
 \int_{\partial^*E_\rho}
 |H_{\bar z_\rho,\rho}-\overline V_\rho|\dHn
 &\le \Per(E_\rho)(1-\overline V_\rho)
 +\int_{\partial^*E_\rho}|H_{\bar z_\rho,\rho}-1|\dHn.
\end{align*}
The first term above, which takes into account the error of replacing
$\overline V_\rho$ by $1$, satisfies
\begin{align*}
 \Per(E_\rho)(1-\overline V_\rho)
 &=\Per(E_\rho)-\Per(B_\rho)
 \le\frac{D_I(t(\rho))}{2n\omega_n\rho^{n-1}}
 \le C(n,\rho_*)D_I(t(\rho)).
\end{align*}

As for the second term, set
$
 a\eqdef\frac{|x-\bar z_\rho|}{\rho}
$,
$
 b\eqdef\frac{x-\bar z_\rho}{|x-\bar z_\rho|}\cdot\nu_\rho,
$
so that $H_{\bar z_\rho,\rho}=ab$, $|b|\le1$, and
\begin{align*}
 |H_{\bar z_\rho,\rho}-1|
 &=|ab-1|=|(a-1)b+(b-1)|
 \le |a-1||b|+1-b\\
 &\le\left|\frac{|x-\bar z_\rho|}{\rho}-1\right|
 +\frac12\left|\nu_\rho-
 \frac{x-\bar z_\rho}{|x-\bar z_\rho|}\right|^2.
\end{align*}
Here we used that both vectors in the last line have unit length, and hence
$1-b$ is one half of their squared distance. Consequently,
\[
 \int_{\partial^*E_\rho}|H_{\bar z_\rho,\rho}-1|\dHn
 \le \mathcal I_1+\mathcal I_2,
\]
where
\begin{align*}
 \mathcal I_1\eqdef\int_{\partial^*E_\rho}|a-1|\dHn, \quad 
 \mathcal I_2\eqdef\frac12\int_{\partial^*E_\rho}
 \left|\nu_\rho-\frac{x-\bar z_\rho}{|x-\bar z_\rho|}\right|^2\dHn.
\end{align*}
Estimate~\eqref{eq.local_FJ_normal} directly gives
\begin{equation}\label{eq.I2_estimate}
 \mathcal I_2\le C(n,R,\rho_*)D_I(t(\rho)).
\end{equation}

For $\mathcal I_1$, we use again the identity
$\displaystyle
 1=b+\frac12\left|\nu_\rho-
 \frac{x-\bar z_\rho}{|x-\bar z_\rho|}\right|^2
$
to write
\begin{align*}
 \mathcal I_1
 &=\underbrace{\int_{\partial^*E_\rho}|a-1|
 \frac{x-\bar z_\rho}{|x-\bar z_\rho|}\cdot\nu_\rho\dHn}
 _{\eqdef\mathcal I_{1,1}}
 +
 \underbrace{\frac12\int_{\partial^*E_\rho}|a-1|
 \left|\nu_\rho-\frac{x-\bar z_\rho}{|x-\bar z_\rho|}\right|^2\dHn}
 _{\eqdef\mathcal I_{1,2}}.
\end{align*}

The first integral can be estimated with Gauss--Green. Consider the vector
field
\[
 X(x)\eqdef\left|\frac{|x-\bar z_\rho|}{\rho}-1\right|
 \frac{x-\bar z_\rho}{|x-\bar z_\rho|}.
\]
Its divergence is locally integrable and equals
\[
 \ddiv X(x)=
 \begin{cases}
 \dfrac n\rho-\dfrac{n-1}{|x-\bar z_\rho|},
 &|x-\bar z_\rho|>\rho,\\[8pt]
 \dfrac{n-1}{|x-\bar z_\rho|}-\dfrac n\rho,
 &|x-\bar z_\rho|<\rho.
 \end{cases}
\]
On the other hand $X\big|_{\partial B_\rho}=0$, so that
\[
  \int_{\partial B_\rho}X\cdot\nu_{B_\rho}\dHn=\int_{B_\rho}\ddiv X\dd x=0,
\]
allowing us to write
\begin{equation}\label{eq.I11_estimate}
  \begin{aligned}
    \mathcal{I}_{1,1}&=
    \int_{E_\rho}\ddiv X(x)\dd x-\int_{B_\rho}\ddiv X\dd x \\
    &=\int_{E_\rho\setminus B_\rho}\ddiv X(x)\dd x-\int_{B_\rho\setminus E_\rho}\ddiv X\dd x\\
    &=\int_{E_\rho\setminus B_\rho}\left(\frac{n}{\rho}-\frac{n-1}{|x-z|}\right)\dd x
    +\int_{B_\rho\setminus E_\rho}\left(\frac{n}{\rho}-\frac{n-1}{|x-z|}\right)\dd x\\
    &\le\frac{n}{\rho}\bigl|E_\rho\,\Delta\,B_\rho(\overline{z}_\rho)\bigr|
    \le C(n,R,\rho_*)D_I(t(\rho))^{1/2}.
  \end{aligned}
\end{equation}

For the second integral, $|x-\bar z_\rho|\le2R$ on $E_\rho$, and thus
\begin{align}
 \mathcal I_{1,2}
 &\le\frac{2R+1}{2\rho_*}
 \int_{\partial^*E_\rho}
 \left|\nu_\rho-\frac{x-\bar z_\rho}{|x-\bar z_\rho|}\right|^2\dHn\notag\\
 &\le C(n,R,\rho_*)D_I(t(\rho)).
 \label{eq.I12_estimate}
\end{align}
Combining these estimates and using $D_I(t(\rho))\le\eta_I\le1$, we obtain
\begin{align*}
 \int_{\partial^*E_\rho}
 |H_{\bar z_\rho,\rho}-\overline V_\rho|\dHn
 &\le C(n,R,\rho_*)
 \bigl(D_I(t(\rho))^{1/2}+D_I(t(\rho))\bigr)\\
 &\le C(n,R,\rho_*)D_I(t(\rho))^{1/2}.
\end{align*}
Squaring proves the first integral estimate in item~(3).

It remains to compare the centers $z_\rho$ and $\bar z_\rho$. Let
$d=|z_1-z_2|$. For $0\le d\le2\rho$, the exact lens formula gives
\[
 |B_\rho(z_1)\Delta B_\rho(z_2)|
 =4\omega_{n-1}\int_0^{d/2}(\rho^2-s^2)^{(n-1)/2}\dd s.
\]
Using this formula for $d\le\rho$, and monotonicity for $d\ge\rho$, yields
\begin{equation}\label{eq.lens_lower_bound}
 |B_\rho(z_1)\Delta B_\rho(z_2)|
 \ge c_{\rm lens}(n)\rho^{n-1}\min\{|z_1-z_2|,\rho\},
 \quad
 c_{\rm lens}(n)=2\omega_{n-1}\left(\frac34\right)^{(n-1)/2}.
\end{equation}
From the minimality of $z_\rho$ and~\eqref{eq.local_FJ_volume}, we get
\begin{align*}
 c_{\rm lens}(n)\rho^{n-1}
 \min\{|z_\rho-\bar z_\rho|,\rho\}
 &\le|B_\rho(z_\rho)\Delta B_\rho(\bar z_\rho)| 
 \le|E_\rho\Delta B_\rho(z_\rho)|
 +|E_\rho\Delta B_\rho(\bar z_\rho)|\\
 &\le2|E_\rho\Delta B_\rho(\bar z_\rho)| 
 \le C(n,R,\rho_*)D_I(t(\rho))^{1/2}.
\end{align*}
Reducing the computable threshold $\eta_I(n,R,\rho_*)$ if necessary, the
minimum is attained by its first branch. Hence
\[
 |z_\rho-\bar z_\rho|
 \le C(n,R,\rho_*)D_I(t(\rho))^{1/2}.
\]
After one further reduction of $\eta_I$, the identity
$\Per(E_\rho)^2=\Per(B_\rho)^2+D_I(t(\rho))$ also gives
$\Per(E_\rho)\le2\Per(B_\rho)$. Consequently,
\begin{align*}
 \int_{\partial^*E_\rho}
 |H_{z_\rho,\rho}-H_{\bar z_\rho,\rho}|\dHn
 &\le\frac{\Per(E_\rho)}{\rho}|z_\rho-\bar z_\rho| \le C(n,R,\rho_*)D_I(t(\rho))^{1/2}.
\end{align*}
Finally, another application of the triangle inequality gives
\begin{align*}
 \int_{\partial^*E_\rho}|H_{z_\rho,\rho}-\overline V_\rho|\dHn
 &\le\int_{\partial^*E_\rho}|H_{z_\rho,\rho}-H_{\bar z_\rho,\rho}|\dHn 
 +
 \int_{\partial^*E_\rho}|H_{\bar z_\rho,\rho}-\overline V_\rho|\dHn\\
 &\le C(n,R,\rho_*)D_I(t(\rho))^{1/2}.
\end{align*}
Squaring finishes the proof of item~(3).
\end{proof}

We now combine the exact formula for the Saint--Venant deficit with the bounds for the derivative of  $q$. The effective confined strong isoperimetric estimate connects these two parts.

\begin{theorem}[Bounded sharp Saint--Venant stability]
\label{thm:boundedSV}
For every $n\ge 2$ and $R\ge 2$ there is a computable constant $C(n,R)>0$ such that every
open $\Omega\subset B_R$ with $|\Omega|=\omega_n$ satisfies
\begin{equation}\label{eq:boundedSV}
   \mathcal{A}(\Omega)^2\le C(n,R) \delta_{\SV}(\Omega).
\end{equation}
\end{theorem}
\begin{proof}
First we consider the case that $\partial \Omega$ is smooth for simplicity. We start by fixing the following quantities
\[
  \rho_\star \eqdef 3/4,
  \quad
  \tau_0 \eqdef \frac{\rho_\star}{2n},
\]
and choose the computable threshold for the isoperimetric deficit $\eta_I=\eta_I(n,R,\rho_\star)>0$ from item~(3) of Proposition~\ref{prop.estimates_q'_DH_DI}.

We first make the endpoint $\rho=1$ precise.  The torsion function is
positive almost everywhere in $\Omega$, hence
$E_\rho\uparrow\Omega$ up to a null set as $\rho\uparrow 1$, it then holds that
\[
 q(1) \eqdef \inf_{z\in\mathbb{R}^n}|\Omega\Delta B_1(z)| = \omega_n\mathcal{A}(\Omega).
\]
Indeed, as $q$ is Lipschitz over $(0,1)$ from Lemma~\ref{lemma.moving_ball_properties}, it extends continuously to $[0,1]$, and since the minimizers $z_\rho$ in the infimum defining $q(\rho)$ are all contained in the same compact set, the identity above holds.

Since the derivative estimate~\eqref{eq.estimate_dev_q} is only an upper
bound for $q'$, we work with its positive part
$q'_+(\rho)\eqdef\max\{q'(\rho),0\}$.  Our main goal is to prove that
\begin{equation}\label{eq:bounded-proof-two-L2}
 \int_{\rho_\star}^1
 \left(q(\rho)^2+(q'_+(\rho))^2\right)\dd\rho
 \le C(n,R)\delta_{\SV}(\Omega),
\end{equation}
which will be done in the sequel. Once~\eqref{eq:bounded-proof-two-L2} is established, a simple argument through the fundamental theorem of calculus connects $\mathcal{A}(\Omega)^2$ and $\delta_{\rm SV}(\Omega)$.

Recall that applying the change of variables $t=t(\rho)$ in the exact identity from
Theorem~\ref{thm.SaintVenant_qualitative} gives
\[
 \delta_{\SV}(\Omega)=
 \int_0^1\frac{D_H(t(\rho))+D_I(t(\rho))}{n^2\omega_n\rho^{n-2}}\dd\mu(\rho).
\]
Since the defects are nonnegative and $\rho\in[\rho_\star,1]$,
\begin{equation}\label{eq:bounded-proof-budget}
 \int_{\rho_\star}^{1}
 \bigl(D_H(t(\rho))+D_I(t(\rho))\bigr)\dd\mu(\rho)
 \le C(n,\rho_\star)\delta_{\SV}(\Omega).
\end{equation}

We divide $[\rho_\star,1]$ into the good set
\[
 G\eqdef\{\rho\in[\rho_\star,1]:
 w(\rho)\ge\tau_0,\ D_I(t(\rho))\le\eta_I\}
\]
and its complement.  The latter has small measure:
\begin{equation}\label{eq:bounded-proof-exceptional}
 |[\rho_\star,1]\setminus G|
 \le C(n,\rho_\star,\eta_I)\delta_{\SV}(\Omega).
\end{equation}
Indeed, Proposition~\ref{prop.various_estimates} gives
\[
 0\le w(\rho)\le\frac{\rho}{n},
 \qquad
 \int_0^1\rho^n\left(\frac{\rho}{n}-w(\rho)\right)\dd\rho
 =\frac{1}{\omega_n}\delta_{\SV}(\Omega).
\]
On $\{w<\tau_0\}$, with $\tau_0=\rho_\star/(2n)$, the integrand is at
least $\rho_\star^{n+1}/(2n)$.  This controls the measure of the low-density
radii.  On the other exceptional set, where $w\ge\tau_0$ but
$D_I>\eta_I$, we use $\dd\mu=w(\rho)\dd\rho$ and estimate~\eqref{eq:bounded-proof-budget} to get
\[
 |\{w\ge\tau_0,\ D_I(t(\rho))>\eta_I\}|
 \le\frac{1}{\tau_0\eta_I}
 \int_{\rho_\star}^1D_I(t(\rho))\dd\mu(\rho)
 \le C\delta_{\SV}(\Omega),
\]
yielding~\eqref{eq:bounded-proof-exceptional}.

We now bound $q$ and $q'_+$ on $G$.   Item~(2) of
Proposition~\ref{prop.estimates_q'_DH_DI} gives
\[
 q(\rho)^2\le C(n)D_I(t(\rho)).
\]
At almost every differentiability point of $q$, we have the estimate~\eqref{eq.estimate_dev_q_3terms} for $q'$. And since now we are inside the good set, using items~(1) and~(3) of Proposition~\ref{prop.estimates_q'_DH_DI}, as
well as the preceding bound for $q$, we obtain
\begin{align*}
 q'(\rho)
 &\le\frac{n}{\rho} q(\rho)
 +\int_{\partial^*E_\rho}|V_\rho-\overline V_\rho|\dHn
 +\int_{\partial^*E_\rho}
       |\overline V_\rho-H_{z_\rho,\rho}|\dHn\\
 &\le C(n,R,\rho_\star)
 \left(D_H(t(\rho))^{1/2}+D_I(t(\rho))^{1/2}\right).
\end{align*}
Since the right-hand side is nonnegative, the same estimate holds for
$q'_+$. Consequently, for almost every $\rho\in G$,
\begin{equation}\label{eq:bounded-proof-good-radius}
 q(\rho)^2+(q'_+(\rho))^2
 \le C\bigl(D_H(t(\rho))+D_I(t(\rho))\bigr).
\end{equation}
But since $w\ge\tau_0$ on $G$, we have $\dd\rho\le\tau_0^{-1}\dd\mu$, so that equation~\eqref{eq:bounded-proof-budget} yields
\[
 \int_G\left(q(\rho)^2+(q'_+(\rho))^2\right)d\rho
 \le C\delta_{\SV}(\Omega).
\]

On the other hand, using the elementary estimates from Lemma~\ref{lemma.moving_ball_properties}
 \[
 q(\rho)\le2\omega_n,
 \qquad q'_+(\rho)\le2n\omega_n
 \quad\text{for almost every }\rho,
 \]
together with \eqref{eq:bounded-proof-exceptional}, these imply the following estimate on the exceptional set
 \[
 \int_{[\rho_\star,1]\setminus G}
 \left(q(\rho)^2+(q'_+(\rho))^2\right)d\rho
 \le C\delta_{\SV}(\Omega).
 \]
Adding the good and exceptional parts proves~\eqref{eq:bounded-proof-two-L2}.

It remains only to turn this averaged estimate into endpoint control. For every $s\in[\rho_\star,1]$, we apply the fundamental theorem of calculus, $q(1) =  q(s)+\int_s^1q'(\rho)\dd\rho$, and taking the average in $s$ over $[\rho_\star,1]$ so that
\begin{align*}
  q(1)
  &=
  \frac{1}{1-\rho_\star}\int_{\rho_\star}^1 q(s)\dd s
  +
  \frac{1}{1-\rho_\star}
  \int_{\rho_\star}^1\int_s^1 q'(\rho)\dd \rho\dd s\\
  &=
  \frac{1}{1-\rho_\star}\int_{\rho_\star}^1 q(s)\dd s
  +
  \frac{1}{1-\rho_\star}
  \int_{\rho_\star}^1
  (\rho-\rho_\star)q'(\rho)\dd\rho.
\end{align*}
Since $q'\le q'_+$ and
$0\le \frac{\rho-\rho_\star}{1-\rho_\star}\le 1$, it follows that
\[
q(1)
\le
\frac{1}{1-\rho_\star}\int_{\rho_\star}^1 q(s)\dd s
+
\int_{\rho_\star}^1 q'_+(\rho)\dd \rho.
\]
So applying the Cauchy--Schwarz inequality to both terms, along with~\eqref{eq:bounded-proof-two-L2} gives
\begin{align*}
q(1)
&\le
\frac{1}{\sqrt{1-\rho_\star}}
\left(\int_{\rho_\star}^1 q(s)^2\dd s\right)^{1/2}
+
\sqrt{1-\rho_\star}
\left(\int_{\rho_\star}^1 (q'_+(\rho))^2\dd\rho\right)^{1/2}\\
&\le
C\sqrt{\delta_{\mathrm{SV}}(\Omega)}
\end{align*}
Since $q(1)=\omega_n\mathcal{A}(\Omega)$, squaring both sides gives
$\mathcal{A}(\Omega)^2\le C(n,R)\delta_{\SV}(\Omega)$.

Now we consider a general open set $\Omega \subset B_R$. By smooth approximations there is a sequence ${(\Omega_k)}_{k \in \mathbb{N}}$ with smooth boundary such that
\[
    \Omega_k \subset \Omega_{k + 1} \subset \Omega, \quad
    v_k \eqdef |\Omega_k|  \xrightarrow[k \to \infty]{} \omega_n, \quad
    |\Omega \setminus \Omega_k|  \xrightarrow[k \to \infty]{} 0, \quad
    \text{and} \quad
    T(\Omega_k)  \xrightarrow[k \to \infty]{} T(\Omega).
\]
Since now $\Omega_k$ does not have volume $\omega_n$, we can rescale it as
\[
    \hat \Omega_k
    \eqdef
    {\left(\frac{\omega_n}{v_k}\right)}^{1/n} \Omega_k, \text{ so that }
    T(\hat \Omega_k)
    =
    {\left(\frac{\omega_n}{v_k}\right)}^{(n+2)/n} T(\Omega_k)
    \xrightarrow[k \to \infty]{}
    T(\Omega).
\]

For all sufficiently large $k$, the dilation factor is at most $2$, and hence
$\hat\Omega_k\subset B_{2R}$. Moreover,
$\mathbf1_{\hat\Omega_k}\to\mathbf1_\Omega$ in $L^1$, so
\[
 \mathcal A(\hat\Omega_k)\longrightarrow\mathcal A(\Omega),
 \qquad
 \delta_{\rm SV}(\hat\Omega_k)\longrightarrow\delta_{\rm SV}(\Omega),
\]
while convergence of the Fraenkel asymmetries follows directly from the $L^1$ convergence of indicator functions, convergence of the torsion can be easily proven with the dual formulation of $\Omega \mapsto T(\Omega)$ from~\eqref{eq.torsion_dual_formulation}. Applying the smooth result in $B_{2R}$ and passing to the limit
gives
\[
 \mathcal A(\Omega)^2\le C(n,2R)\delta_{\rm SV}(\Omega).
\]
Since the final constant is a result of finitely many computable operations, proves the theorem.
\end{proof}

\section{Removing boundedness assumptions}\label{sec.sharp_SV_unbounded}

In this section, we provide a truncation argument to reduce the stability question to sets inside a fixed ball. The argument we provide mirrors the one of Brasco--De Philippis--Velichkov~\cite[Lemma~5.3]{brasco2015faber}, with the exception that we use the global suboptimal stability estimate with exponent $3$ from~\ref{thm.SaintVenant_qualitative}. Albeit the similarity, we include the
argument for completeness and for the sake of verifying the computability of every constant.

\begin{lemma}\label{lemma.truncation}
There are computable dimensional constants $C_{\rm T}(n)>0$, $d(n)\ge2$,
and $\delta_*(n)\in(0,1]$ such that, if $|\Omega|=\omega_n$ and $\delta_{\rm SV}(\Omega)\le\delta_*(n)$, then there is an open set $\tilde\Omega$ with $|\tilde\Omega|=\omega_n$ and
\begin{equation}\label{eq.truncation}
 \diam(\tilde\Omega)\le d(n),\qquad
 \delta_{\rm SV}(\tilde\Omega)
 \le C_{\rm T}(n)\delta_{\rm SV}(\Omega),\qquad
 \mathcal A(\Omega)\le\mathcal A(\tilde\Omega)
 +C_{\rm T}(n)\delta_{\rm SV}(\Omega).
\end{equation}
\end{lemma}
\begin{proof}
We follow the direct proof of \cite[Lemma~5.3]{brasco2015faber}, replacing the quartic stability from~\cite{FuscoMaggiPratelli2009Eigenvalue} by Theorem~\ref{thm:suboptimal_stabilitySV}. Recall the torsion functional's dual formulation given by
\[
 \mathcal{E}(\Omega)\eqdef\min_{v\in W_0^{1,2}(\Omega)}
 \left\{\frac{1}{2}\int_\Omega|\nabla v|^2\dd x-\int_\Omega v\dd x\right\}
 =-\frac{1}{2}T(\Omega).
\]
The Euler-Lagrange equations for the problem above imply that the torsion function $u = u_\Omega$ is optimal for $\mathcal{E}(\Omega)$, and the deficit can be rewritten as 
\[
  \frac{1}{2}\delta_{\rm SV}(\Omega) 
  = 
  \mathcal{E}(\Omega) - \mathcal{E}(B_1).
\]

If $\delta_{\rm SV}(\Omega)=0$, Theorem~\ref{thm:suboptimal_stabilitySV} gives
$\mathcal A(\Omega)=0$, and we take $\tilde\Omega=B_1$, hence we assume $\delta_{\rm SV}(\Omega)>0$. Up to translating the set $\Omega$, we can assume that an optimal asymmetry ball is given by $B_1$, so defining 
\[
 b_k\eqdef\frac{|\Omega\setminus B_k|}{\omega_n},\qquad k\ge1,
\]
the suboptimal stability for Saint-Venant from Theorem~\ref{thm:suboptimal_stabilitySV} gives
\begin{equation}\label{eq.truncation_seed}
 b_1\le\mathcal A(\Omega)
 \le \bigl(C_3(n)\delta_{\rm SV}(\Omega)\bigr)^{1/3}.
\end{equation}

Next, we use the following tail estimate from~\cite[Lemma~5.1]{brasco2015faber}: there exists a computable dimensional constant $L_n$ such that 
\begin{equation}\label{eq.torsion_tail}
 \|u\|_{L^\infty(\Omega\setminus B_{R+1})}
 \le L_n |\Omega \setminus B_R|^{1/n}. 
\end{equation}
In particular, taking $R = k$ this estimate implies that $\|u\|_{L^\infty(\Omega\setminus B_{k+1})} \le L_n b_k^{1/n}$. 

Let $\varphi_k(x)=\min\{1,(k+2-|x|)_+\}$ and use $\varphi_ku$ as a competitor for $\mathcal{E}(\Omega\cap B_{k+2})$. The optimality of $u$ in the definition of $\mathcal{E}(\Omega)$ gives
\begin{align*}
 \mathcal{E}(\Omega\cap B_{k+2})
 &\le\mathcal{E}(\Omega)
 +\frac{1}{2}|\Omega\setminus B_{k+1}|
 \left(\|u\|_{L^\infty(\Omega\setminus B_{k+1})}
 +\|u\|_{L^\infty(\Omega\setminus B_{k+1})}^2\right).
\end{align*}
 Since $b_k\le1$, \eqref{eq.torsion_tail} gives
\begin{equation}\label{eq.truncation_energy}
 \mathcal{E}(\Omega\cap B_{k+2})
 \le\mathcal{E}(\Omega)+A'_n b_k^{1+1/n},
\end{equation}
with computable $A'_n$. 

Now we want to apply the qualitative Saint-Venant inequality to the set $\Omega\cap B_{k+2}$, recalling the scaling for the torsion $T(t \Omega') = t^{n+2}T(\Omega')$, the minimality of the ball gives
\[
 \mathcal{E}(B_1)(1-b_{k+2})^{\frac{n+2}{n}}
 \le\mathcal{E}(\Omega\cap B_{k+2})
 \le\mathcal{E}(B_1)+\frac{1}{2}\delta_{\rm SV}(\Omega)+A'_n b_k^{1+1/n}.
\]
Since $\mathcal{E}(B_1)<0$ and $1-(1-b)^p\ge b$ for $b\in[0,1]$, this yields
\begin{equation}\label{eq.truncation_recursion}
 b_{k+2}\le K_n\bigl(\delta_{\rm SV}(\Omega)+b_k^{1+1/n}\bigr),
\end{equation}
where $K_n$ is computable from $A'_n$, $\omega_n$, and
$T(B_1)=\omega_n/[n(n+2)]$.

Put $q=1+1/(2n)$. Choose a positive rational $\delta_*(n)\le1$ so small that
\[
 \bigl(C_3(n)\delta_*(n)\bigr)^{1/3}<1,
 \qquad
 2K_n\bigl(C_3(n)\delta_*(n)\bigr)^{1/(6n)}\le1.
\]
As long as $b_{k+2}\ge2K_n\delta_{\rm SV}(\Omega)$, recursion
\eqref{eq.truncation_recursion} implies $b_{k+2}\le b_k^q$. Choose a
computable integer $m(n)$ such that $q^m>3$  and
  \[
 C_3(n)^{q^{m}/3}\delta_*(n)^{q^{m}/3-1}<2K_n.
\]
Iteration separately on the even and odd indices shows that there is
$k_0\le2m(n)+2$ such that
\begin{equation}\label{eq.truncation_tail_mass}
 b_{k_0+1}<2K_n\delta_{\rm SV}(\Omega).
\end{equation}
All these choices are finite searches over rational numbers and integers.

  Set $U=\Omega\cap B_{k_0+3}$ and
$r=(|U|/\omega_n)^{1/n}$. By reducing $\delta_*(n)$ once more,
\eqref{eq.truncation_tail_mass} gives $1/2\le r^n\le1$. Define
$\tilde\Omega=r^{-1}U$. Then $|\tilde\Omega|=\omega_n$ and
\[
 \diam(\tilde\Omega)\le2^{1+1/n}(2m(n)+5)\eqdef d(n).
\]
Equations \eqref{eq.truncation_energy} and
\eqref{eq.truncation_tail_mass}, followed by the scaling of $\mathcal{E}$,
give
\begin{align*}
 \delta_{\rm SV}(\tilde\Omega)
 &=2\left(r^{-(n+2)}\mathcal{E}(U)-\mathcal{E}(B_1)\right)\\
 &\le r^{-(n+2)}
 \left(\delta_{\rm SV}(\Omega)+2A'_n b_{k_0+1}^{1+1/n}\right)
 \le C(n)\delta_{\rm SV}(\Omega).
\end{align*}
Finally, let $B_1(z)$ be optimal for $\tilde\Omega$. Scaling this ball back
and using \eqref{eq.truncation_tail_mass},
\begin{align*}
 \omega_n\mathcal A(\Omega)
 &\le|\Omega\setminus U|+|U\Delta B_r(rz)|
   +|B_r(rz)\Delta B_1(rz)|\\
 &\le\omega_n\mathcal A(\tilde\Omega)+C(n)\delta_{\rm SV}(\Omega).
\end{align*}
Enlarging the computable constant $C(n)$ gives~\eqref{eq.truncation}.
\end{proof}

\begin{proof}[Proof of Theorem~\ref{thm.sharp_Saint_Venant_unbounded}]
Assume $|\Omega|=\omega_n$, and abbreviate
$\delta=\delta_{\rm SV}(\Omega)$, $C_{\rm T}=C_{\rm T}(n)$,
$d=d(n)$, and $C_{\rm b}=C(n,d)$, where $C(n,d)$ is the computable constant
 from Theorem~\ref{thm:boundedSV}.

If $\delta\ge\delta_*(n)$, then $\mathcal A(\Omega)^2\le4$ gives
\[
 \mathcal A(\Omega)^2\le\frac{4}{\delta_*(n)}\delta.
\]
  Suppose that $\delta<\delta_*(n)$. Lemma~\ref{lemma.truncation} gives, up to
a translation, a set $\tilde\Omega\subset B_d$ such that
\[
 \delta_{\rm SV}(\tilde\Omega)\le C_{\rm T}\delta,
 \qquad
 \mathcal A(\Omega)\le\mathcal A(\tilde\Omega)+C_{\rm T}\delta.
\]
  The bounded theorem therefore yields
\[
 \mathcal A(\Omega)
 \le\sqrt{C_{\rm b}C_{\rm T}\delta}+C_{\rm T}\delta.
\]
Since $\delta<\delta_*(n)$,
\[
 \mathcal A(\Omega)^2
 \le\bigl(2C_{\rm b}C_{\rm T}
 +2C_{\rm T}^2\delta_*(n)\bigr)\delta.
\]
Thus the theorem holds with the computable constant
\begin{equation}\label{eq.global_CSV}
 C_{\rm SV}(n)
 \eqdef\max\left\{
 \frac{4}{\delta_*(n)},
 2C(n,d(n))C_{\rm T}(n)+2C_{\rm T}(n)^2\delta_*(n)
 \right\}.
\end{equation}
This concludes the proof.
\end{proof}


\section{Final comments}\label{sec.comments}

\subsection{Other types of eigenvalues}\label{sec.generalized_eigenvalues}

As done in~\cite{brasco2015faber}, we can more generally define the first eigenvalue of the Laplacian in the $L^q$-topology, for all
\[
    1 \le q < 2^* 
    \eqdef 
    \begin{cases}
        \frac{2n}{n-2},& \text{ if } n \ge 3, \\ 
        +\infty,& \text{ otherwise,}
    \end{cases}
\]
as the following quantity 
\begin{equation}
     \lambda_{2,q}(\Omega)
      \eqdef
      \inf_{
        \substack{
            v\in W_0^{1,2}(\Omega),\\
            \norm{v}_{L^q(\Omega)} = 1
        } \, 
      }
      \displaystyle
      \int_\Omega |\nabla v|^2\dd x .
\end{equation}
Setting $\theta(q,n)=2/n+2/q-1>0$, the natural scaling of this new quantity is given by 
\[
 \lambda_{2,q}(t\Omega)=t^{-n\theta(q,n)}\lambda_{2,q}(\Omega).
\]
Thus the natural scale-invariant deficit is defined as
\begin{equation}\label{eq.deficit_2q}
 \delta_{2,q}(\Omega)
 \eqdef |\Omega|^{\theta(q,n)}\lambda_{2,q}(\Omega)
 -|B_1|^{\theta(q,n)}\lambda_{2,q}(B_1)\ge0.
\end{equation}
For $q=2$, this is precisely the Faber--Krahn deficit.

Thanks to the Dirichlet energy being kept the same, this quantity is also related to the same torsion functional $T(\Omega)$, thanks to the Kohler-Jobin type inequality proven by Brasco~\cite{brasco2014torsional}: 
\begin{equation}\label{eq.kohler_jobin_generalized}
      \lambda_{2,q}(\Omega)T(\Omega)^{\alpha(q,n)}
      \ge
      \lambda_{2,q}(B_1)T(B_1)^{\alpha(q,n)},
      \text{ where }
      \alpha(q,n) \eqdef 
      \frac{\frac{2}{n} + \frac{2}{q} - 1}{\frac{2}{n} + 1}. 
\end{equation}

Applying
i$1-(1+x)^{-1/\alpha}\le x/\alpha$ to
\eqref{eq.kohler_jobin_generalized}, exactly as in the proof of
Theorem~\ref{thm.sharp_FK}, and then scaling gives
\[
 \mathcal A(\Omega)^2\le C_{2,q}(n)\delta_{2,q}(\Omega),
 \qquad
 C_{2,q}(n)=\frac{C_{\rm SV}(n)T(B_1)}
 {\alpha(q,n)\lambda_{2,q}(B_1)|B_1|^{\theta(q,n)}}.
\]
This is the sharp estimate of Brasco--De Philippis--Velichkov~\cite{brasco2015faber}, but with a
computable constant, since $\lambda_{2,q}(B_1)$ is computable from its radial
Euler--Lagrange equation.

The sharp quantitative Faber--Krahn inequality for the first Dirichlet eigenvalue of the $p$-Laplacian was established by Fusco--Zhang~\cite{FuscoZhang2017}; see also Bhattacharya~\cite{Bhattacharya2001} for earlier asymmetry estimates. The transfer above, however, is specific to the quadratic Dirichlet energy and the linear torsion problem, and therefore does not immediately provide effective constants when $p\ne2$. We expect that the level-set techniques developed here may also be useful in that setting, which we plan to address in future work.

\subsection{Different proof strategies}\label{sec.alternative_routes}

While interacting with GPT 5.4 and 5.5 about this problem, several other proof ideas were explored, as alluded to previously. We record four routes that we explored in our path to obtain the proof above of the effective stability. We note that, currently, none of the proofs have been closed, and hence they have not yet produced an effective proof with the same generality as the argument above. The first is a quantitative version of the known selection proof, whereas the remaining three use parts of the level-set structure in different ways.

\subsubsection{An effective selection principle.}
One could try to make the proof of Brasco--De Philippis--Velichkov
\cite{brasco2015faber} effective. Starting from one set with small torsional
deficit, a penalized minimization problem should produce a quasi-minimizer
whose asymmetry and deficit-to-asymmetry ratio remain comparable to those of
the original set. The free-boundary theory of Alt--Caffarelli
\cite{AltCaffarelli1981} would then place the selected boundary in a graph
class. A quantitative bootstrap followed by the nearly spherical
second-variation estimate of~\cite{brasco2015faber}, in the spirit of
Fuglede~\cite{Fuglede1989}, would give the desired quadratic bound.

The delicate point is not the formal penalization, but rather that, in the known argument,
compactness is used to obtain a global graph over one sphere and the strong convergence needed for the perturbative expansion. A direct version would require quantitative graph-entry, patching, and Schauder estimates, all uniform for a single selected minimizer and stable under volume normalization. It does seem, to us, like this could be the most viable route to obtain an alternative proof of our main result, but since we were seeking a proof that uses all geometric information of the problem at once, we decided to pursue our proof strategy rather than that one. 

In contrast to the effective selection principle route, the one developed here has the added bonus of not having to rely too much on any form of regularity theory, being mostly elementary when confined to basic theorems in geometric measure theory. 

\subsubsection{A metric trace for the level-set foliation.}
A second possible route could retain the whole family $E_\rho$, but regard instead
$F_\rho=\rho^{-1}E_\rho$ as a curve of fixed-volume sets modulo
translations, with
\[
 d_{\mathcal X}([A],[C])
 \eqdef \inf_{y\in\mathbb R^n}|A\mathbin\Delta(C+y)|.
\]
The term $D_I$ should control the squared distance of $F_\rho$ from the unit
ball through strong isoperimetry and the homothetic part of its motion. The
term $D_H$ should control the remaining normal-velocity component. Together
they would bound the metric speed of $\rho\mapsto[F_\rho]$. An action
estimate followed by a trace inequality on an outer interval would then
yield a near-boundary level close to a ball, and the omitted boundary layer
could be added back at the end.

This formulation is close to the proof above, but asks for more. The deficit
controls the natural action with respect to the weighted measure
$\dd\mu=-t'(\rho)\dd\rho$, whereas transport to the endpoint requires
unweighted control across radii where $-t'(\rho)$ is small. The small measure
of these radii does not control the metric variation of a rough set-valued
curve. The present proof avoids this obstruction by estimating only the
positive derivative of the scalar moving-ball distance $q$ and by using its
global Lipschitz bound on the exceptional radii. 

\subsubsection{Selection followed by an inner Serrin level.}
An intermediate possibility is to use selection only to obtain uniform
$C^{1,\gamma}$ free-boundary estimates. For the torsion function of the
selected set, the level-set identity produces a regular level close to the
boundary on which $D_H$ and $D_I$ are small. Uniform H\"older control of
$\nabla u$ could upgrade the weighted $L^2$ oscillation measured by $D_H$ to
an $L^\infty$ approximate Serrin condition. A quantitative Serrin estimate,
for instance in the spirit of
\cite{BrandoliniNitschSalaniTrombetti2008}, would confine that level between
concentric spheres. The inherited $C^{1,\gamma}$ control could then promote
it to a small radial graph, where a direct perturbative estimate applies.

This route may avoid the $C^{2,\gamma}$ bootstrap required by the direct
selection proof, but it does not avoid selection or free-boundary
regularity. It also requires uniform constants in the interpolation from the
weighted boundary defect, in the Serrin estimate, and in the passage from an
annular $C^{1,\gamma}$ hypersurface to a radial graph. These points are
not to be overlooked, particularly in higher dimension and for disconnected or
multiply connected selected sets, for which reasons it ended up being abandoned. 

\subsubsection{A planar inset and radial-graph reduction.}
In dimension two, one can seek a selection-free proof from a single inner
level. After choosing a level at distance comparable to the square root of
the deficit, one keeps its largest connected component. The isoperimetric
defect should trap this inset in a thin annulus. Planar crossing identities
can then measure the rays on which the inset folds over itself. Cutting the
corresponding outer caps, rather than filling the intervening gaps, is
expected to produce a small-amplitude radial graph. Fourier or Steklov
coercivity for that graph would conclude the planar estimate.

Several intermediate steps remain, however, unresolved in that route. One must control holes meeting the
original boundary, realize the measurable first-crossing profile by an
admissible graph, and estimate the change of torsion under cap cutting.
Measure control alone is insufficient for the last step because a set of
zero area may have positive capacity. The translation mode must also be
removed without destroying radiality. These issues are specific enough to
make the planar route plausible, but they prevent it from being a proof at present, and no comparable all-dimensional construction is known.

\bibliographystyle{plain}
\bibliography{refs}

\end{document}